\documentclass[12pt]{article}
\usepackage{amssymb}
\usepackage{amsfonts}
\usepackage{color}
\usepackage{graphics}
\usepackage{epsfig,psfrag,graphicx}
\usepackage{amssymb}
\usepackage{eepic,epic}

\usepackage{epsfig} % postscript figures
\usepackage{amsmath} % ams math
\usepackage{amssymb} % ams symbols
\usepackage{amsbsy} % ams bold symbols

\title{Multi-dimensional scalar balance laws with discontinuous flux}
\author{Piotr Gwiazda%\footnote{\texttt{pgwiazda@mimuw.edu.pl}}
\and Agnieszka \'Swierczewska-Gwiazda%\footnote{\texttt{aswiercz@mimuw.edu.pl}}\\[2ex]
\and Petra Wittbold
\and Aleksandra Zimmermann}
\date{}
\begin{document}
\maketitle \maketitle
%%%%%%%%%%%%%%%%%%%%%%%%%
\newtheorem{definition}{Definition}[section]
\newtheorem{theorem}{Theorem}[section]
\newtheorem{lemma}[theorem]{Lemma}
\newtheorem{corollary}[theorem]{Corollary}
\newtheorem{proposition}[theorem]{Proposition}
\newtheorem{Rem}[theorem]{Remark}
%%%
%%%%%%%%%%%%%%

\renewcommand{\theequation}{\thesection.\arabic{equation}}
\newcommand{\nc}{\newcommand}
\nc{\V}{{\cal V}} 
\nc{\M}{{\cal M}} 
\nc{\T}{{\cal T}}
\nc{\D}{{\cal D}} 
\nc{\R}{{\mathbb R}} 
\nc{\N}{{\mathbb N}}
\nc{\qed}{\mbox{}\nolinebreak\hfill \rule{2mm}{2mm}} 
\nc{\weak}{\rightharpoonup}
\nc{\weakstar}{\stackrel{\ast}{\rightharpoonup}} 
\nc{\proof}{{\bf Proof: }} 
\def\bbbone{{\mathchoice {\rm 1\mskip-4mu l}
{\rm 1\mskip-4mu l} {\rm 1\mskip-4.5mu l} {\rm 1\mskip-5mu l}}}
\renewcommand{\div}{{\mathrm{div}}\,}
\def\Rdp{\mathbb{R}_+\times \mathbb{R}^{N}}
\def\bG{A}
\def\prob{\mathop{\mathrm{Prob}}\nolimits}
\def\sgn{\mathop{\mathrm{sgn}}\nolimits}
\def\supp{\mathop{\mathrm{supp}}\nolimits}
\def\bF{A}
\def\bQ{Q}

\begin{abstract}
We consider the problem of existence of entropy weak solutions to {scalar balance} laws with a dissipative 
source term. The flux function may be discontinuous with respect both to the space variable $x$ and the unknown quantity~$u$. The problem is formulated in the framework of multi-valued mappings. We use the notion of entropy-measure valued solutions to prove the so-called contraction principle and comparison principle.
 \\[2ex]
\\[2ex]
 \end{abstract}

\section{Introduction}

Our interest is directed to the following Cauchy problem describing the evolution of $u:\R_+\times\R^N\to\R$
\begin{align}\label{E1}
 u_t + \div \Phi(x,u)\ni f(t,x,u)&\quad \mbox{ on } \R_+\times \R^N,\\ \label{E0}
%u=a   &\quad\mbox{ on } \Sigma=(0,T)\times \partial\Omega\\
u(0,\cdot)=u_0 &\quad\mbox{ on }  \R^N.
\end{align}
where $\Phi:\R^N\times\R\to2^{\R^N}$ is a multi-valued mapping and   
$f:\R_+\times\R^N\times\R\to\R$ is a source term. Moreover   $u_0:\R^N\to\R$  is a given  initial data. The assumptions for $\Phi$ and $f$ shall be presented below. The formulation of the problem in the language of multi-valued flux function allows to capture 
relations which are not necessarily functions.
% {discontinuous relations}.
% \textcolor{blue}{To treat the discontinuous function we fill the jumps and identify the graph with a continuous curve what is then described by means of \eqref{E1}.}\marginpar{\textcolor{blue}{the meaning of the sentence is not clear to me}}

{We will assume that the flux function is in the form 
of a composition, which allows, with an appropriate change of variables, %and consequently   
to formulate the definition of entropy weak solutions in terms of the new variables}. An important property of such defined solutions is that in  case of smooth fluxes they correspond to the classical definition of entropy weak solutions, see e.g. Kru\v zkov~\cite{Kr70}. We assume about $\Phi$ and $f$  that:
\begin{enumerate}
\item[(H1)] $\Phi(x,u)$ is a multi-valued mapping given by the formula $\Phi(x,u)=A(\theta(x,u))$ where
$A: \R\to\R^N$, $A$ is continuous  and 
 $\theta : \R^N \times \R \to 2^{\R}\setminus \emptyset $ is a multi-valued mapping such that, for almost every $ x \in \R^N$, $\theta(x,\cdot) : \R \to 2^{\R}\setminus \emptyset $ is a maximal monotone operator with $0\in \theta (x, 0)$.
%$\theta(x,u)$ is $x-$dependent maximal monotone mapping such that its 
The inverse to $\theta$ (w.r.t $u$), which we call 
$\eta$, is continuous. 
 Moreover, we assume that 
\begin{equation}\label{betaas}
\theta^*(\cdot,l) \in L^1(\R^N)
\end{equation}
for each $l \in \R$, where  $\theta^*$ denotes the minimal selection of the graph of $ \theta $.
%Moreover, there exists a measurable selection $\theta^*$.\marginpar{Consider how to formulate the assumption on measurability of $\theta$}
\item[(H2)] there exist continuous functions $h_1$ and $h_2$ with {$\lim_{|u|\to\infty}h_1(u)=\infty$ such that\begin{equation}\label{h1h2}
h_1(u)\le |\overline{\theta}|\le h_2(u) 
\end{equation}
for all $\overline{\theta}\in \theta(x,u)$, almost every $x\in\mathbb{R}^N$ and all $u\in\R$}
\item[(H3)]  there exists $1\le p\le \frac{N}{N-1}$ and  constants $R_\infty >0$ and $C_\infty>0$ such that for all $x>R_\infty$
$$|A(s)|^p\le C_\infty |\eta(x,s)|$$ 
\item[(H4)]  $f(\cdot,\cdot,u)\in L^1_{loc}(\R_+\times\R^N)$ for all $u\in\R$; $f(t,x,\cdot)$ is continuous and $f(t,x,0)=0$ for a.a. $(t,x)\in\R_+\times\R^N$. Moreover $f$ is dissipative ($-f$ is monotone w.r.t. the last variable), i.e.,  
\begin{equation}\label{A4}
(f(t,x,u)-f(t,x,v))(u-v)\le0\quad \mbox{for all } u,v\in \R \mbox{ and a.a. } (t,x)\in \R_+\times \R^N
\end{equation}
%There exists a measurable selection (single-valued function) from $\theta(x,u)$, which we call $\bar\theta(x,u)$.
\end{enumerate}
\begin{Rem}
One could consider {a} more general source term, namely for almost all $(t,x)\in\R_+\times\R^N$ a maximal monotone (possibly multi-valued) mapping $f$. Then we would rewrite \eqref{E1} as $u_t+\div\Phi(x,u)-f(t,x,u)\ni 0$. The scalar conservation laws with a multi-valued source term were considered e.g. in \cite{GwSw2005}.
\end{Rem}
The approach of considering  the flux function in form of a composition was used by Panov in~\cite{Panov} to solve the problem of well-posedness for a scalar conservation law {without source term (i.e. $f=0$) and a} flux function discontinuous with respect to~$x$.
More precisely, the author assumed that $\Phi(x,u)=A(\theta(x,u))$, where $A\in{\cal C}(\R;\R^N)$ and 
$\theta:\R^N\times\R\to\R$ is a Carath\'eodory function, which is for almost all $x\in\R^N$ strictly increasing with respect to $u$. Moreover the same condition as (H2) was assumed. Hence if  $\eta(x,v)$ is the  inverse to $\theta$, i.e., $\theta(x,\eta(x,v))=v$ then $u$ is a solution to \eqref{E1}--\eqref{E0}  {with $f=0$} if there exists $v$ such that  {$u=\eta(x,v)$ and} the following entropy inequality is satisfied in the distributional sense in $\R_+\times\R^N$ for all $k\in\R$
\begin{equation}
|\eta(x,v)-\eta(x,k)|_{t} \!+ \div ({\rm sgn}\  (v-k) (\bG(v)-\bG(k)))\le 0.
\end{equation}

The corresponding approach we find for fluxes discontinuous only with respect to~$u$ in the paper by Carrillo,~\cite{Carrillo2003}.
The author studied  the problem
in a bounded domain
 \begin{equation}
 \begin{split}
 u_t+\div \Phi(u)\ni f&\quad{\rm in}\  (0,T)\times\Omega\\
 u(0)=u_0&\quad{\rm in}\ \Omega
 \end{split}
 \end{equation}
 under the assumption that $\Phi$ is allowed to have  discontinuities of first type on a  finite subset of $\R$.   
 After a change of variables the author deals with the following problem
  \begin{equation}
 \begin{split}
 g(v)_t+\div \Psi(u)=f&\quad{\rm in}\  (0,T)\times\Omega,\\
 g(v(0))=u_0&\quad{\rm in}\ \Omega.
 \end{split}
 \end{equation}
  The proof  of existence of solutions bases upon the comparison principle and the entropy inequality involving a version of  semi Kru\v zkov entropies, namely $E(v,k)=(g(v)-g(k))^+$. 
  
 The  similar problem was considered in  Bul\'{\i}\v{c}ek et al.~\cite{BuGwMaSw2011} 
 with the use of different approach, namely 
  \begin{equation}
 \begin{split}
 u_t+\div \Phi(u)= 0&\quad{\rm in}\  \R_+\times\R^N,\\
 u(0)=u_0&\quad{\rm in}\ \R^N.
 \end{split}
 \end{equation}
 The authors showed existence and uniqueness of entropy weak solutions for jump continuous $\Phi$ (i.e. having countable, not necessarily finite,   number of jumps). For the proof they essentially used the method of entropy measure-valued solutions introduced by DiPerna, cf.~\cite{DiPerna} and later extended by Szepessy in \cite{Sz89a}. To handle the discontinuity of the flux function Bul\'{\i}\v{c}ek et al. showed existence of a parametrization $U$, namely a nondecreasing function such that $\Phi\circ U$ is continuous. 
 %The uniqueness of entropy weak solutions holds up to the level sets of the parametrization $U$.  ROZWINAC. 
 
 These ideas are combined in~\cite{BuGwSw2013}, where the authors treat the case of a flux function discontinuous in $x$ and $u$ for the problem 
\begin{equation}
  \begin{split}
 u_t+\div \Phi(x,u)= 0&\quad{\rm in}\  \R_+\times\R^N,\\
 u(0)=u_0&\quad{\rm in}\ \R^N.
 \end{split}\label{Bu}
 \end{equation}
The set of assumptions corresponds to the one formulated by Panov in~\cite{Panov}, namely 
$\Phi(x,u)=A(\theta(x,u))$ extended by the possibility that $A$ is a jump continuous function. Again through appropriate estimates for entropy measure-valued solutions and finding the pa\-ra\-me\-tri\-za\-tion $U$ the authors showed well-posedness for \eqref{Bu}. Both in \cite{BuGwSw2013}  and  \cite{BuGwMaSw2011} the uniqueness of entropy weak solutions needs to be understood up to the level sets of the pa\-ra\-me\-tri\-za\-tion $U$. This is also related with a restricted family of entropies which are allowed, what we will discuss in more detail after the statement of definition and main theorem. 

In the present paper we have added a source term, which requires additional attention in various crucial estimates. However the main novelty 
is to combine the approaches from \cite{BuGwSw2013} and \cite{Carrillo2003}
 and consequently obtain a stronger result. The proof bases on the combination of comparison principle and formulating the definition with help of the entropies of semi-Kru\v zkov type with compactenss arguments. The approach presented here gives additional advantages. If the starting point are  considerations on the problem formulated with discontinuous flux (jump continuous), we shall first fill up the jumps. In the case of  \cite{BuGwSw2013} we may only do it with intervals, however in the current setting we have more freedom. 
 %(OBRAZEK?)
  We come back to this issue at the end of the introduction, after formulating the definition and recalling in more detail the framework of \cite{BuGwSw2013}.

Before we formulate the definition of entropy weak solutions let us introduce some notation.  By ${\cal D}(\Omega)$ we mean the set of smooth functions with a compact support in $\Omega$, ${\cal C}(\Omega;X)$ is the set of continuous functions from $\Omega$ to the space $X$.  For $ 1\le p\le\infty$ by $L^p(\Omega)$ we understand 
standard Lebesgue spaces and by  $L^p(\R_+;X)$ Bochner spaces. 
 
\begin{definition}\label{weak}
Let $ \Phi, f$ satisfy the assumptions (H1)--(H4). We say that  a function $u\in L^\infty(\R_+\times\R^N)\cap 
L^\infty(\R_+;L^1(\R^N))$ is an entropy weak solution of~\eqref{E1}-\eqref{E0}
if there exists a function $g\in L^\infty(\Rdp)$ such that  {$u=\eta(x,g)$ and} for all $\psi\in{\cal D}(\R\times\R^N),\,\psi\ge0$ and for all $k\in\R$
%\marginpar{Czy pierwsze wyrazenie jest mnozone przez chi}
\begin{itemize}
\item[(i)]
\begin{align}\nonumber
\int_{\Rdp} \left\{\left(\eta(x,g)-\eta(x,k)\right)^+\psi_t+\chi_{\{g>k\}}(A(g)-A(k))\nabla \psi+\chi_{\{g>k\}} f\psi
\right\} \\ 
\ge-\int_\Omega(u_0-\eta(x,k))^+\psi(0,\cdot), 
\end{align}
	\item[(ii)]
\begin{align} \nonumber
\int_{\Rdp} \left\{(\eta(x,k)-\eta(x,g))^+\psi_t+\chi_{\{k>g\}}(A(k)-A(g))\nabla \psi-\chi_{\{k>g\}} f\psi
\right\} \\
\ge-\int_{\R^N}(\eta(x,k)-u_0)^+\psi(0,\cdot). 
\end{align}

\end{itemize}

\end{definition}
 {
\begin{Rem}
Note that $(i)$ and $(ii.)$ of Definition \ref{weak} are equivalent to the conditions
\begin{equation}\label{eqcondI}
\int_{\Rdp} |\eta(x,g)-\eta(x,k)|\psi_t+\operatorname{sgn}(g-k)(A(g)-A(k))\nabla \psi+\operatorname{sgn}(g-k)f\psi \geq 0
\end{equation}
for all $\psi\in\mathcal{D}((0,T)\times\mathbb{R}^N)$ such that $\psi\geq 0$ and
\begin{equation}\label{eqcondII}
\operatorname{ess}\lim_{t\rightarrow 0}\int_K |u(t,x)-u_0| \ dx=0
\end{equation}
for any compact $K\subset\mathbb{R}^N$.
\end{Rem}}
Now we are ready to formulate the main result of the paper on the existence of entropy weak solutions.

\begin{theorem}\label{main}
Let $ \Phi, f$ satisfy the assumptions (H1)--(H4). Assume $u_0\in L^1(\R^N)\cap L^{\infty}(\R^N)$. Then
  there exists an entropy weak  solution $u$ to
\eqref{E1}--\eqref{E0} in the sense of Definition~\ref{weak} 
\end{theorem}

To understand the advantage of the framework presented here let us closely observe the approach in \cite{BuGwSw2013} and recall that by an {\it admissible parametrization}  of $A$ the authors understand
 a couple $({\cal A},U)$ if  the function $U\in \mathcal{C}(\R)$ is nondecreasing and $\lim_{s\to \pm \infty}U(s)=\pm \infty$.  {Moreover, defining}
\begin{equation}\label{alphabeta}
\alpha_k:= \inf_{\alpha; \; U(\alpha)=z_k}\alpha,
\qquad \beta_k:= \sup_{\beta; \; U(\beta)=z_k}\beta
\end{equation}
%
%&\quad \textrm{then for all $k\in \mathbb{N}$ there holds }\alpha_k < \beta_k< \alpha_{k+1},\\
it is required that the function $U$ is constant on $[\alpha_k,\beta_k]$ and strictly increasing
 on $(\beta_k,\alpha_{k+1})$ for all $k\in\mathbb{N}$. 
The function ${\cal A}\in  {\mathcal{C}(\R;\mathbb{R}^N)}$ satisfies ${\cal A}(s)\in A(U(s))$
and  is linear on $[\alpha_k,\beta_k]$ for all $k\in \mathbb{N}$.
%For every jump continuous function $\bG$ there exists such an admissible parametrization, what we discuss in more detail at the beginning of Section~\ref{Emvs}.
%\subsection{Definition of entropy weak solutions and main result}
%Finally, we introduce a notion of an entropy weak solution to \eqref{E1} corresponding to \eqref{E1Bn}--\eqref{KruzBn}.
Then 
$ {u\in L^{\infty}(0,\infty;L^1(\R^N))\cap L^{\infty}(\R_+\times\mathbb{R}^N)}$ is an entropy
weak solution to \eqref{E1} related to $({ A},\theta)$ and $u_0$ for an admissible parametrization $({\cal A},U)$ of $A$ if there exists a function $g\in L^{\infty}(\Rdp)$ such that
\begin{equation}
\eta(x,U(g(t,x)))=u(t,x),\quad {\cal A}(g(t,x))\in \bG(\theta(x,u(t,x))) \qquad \textrm{ a.e. in } \Rdp,\label{u1}
\end{equation}
\begin{align}
& {\operatorname{ess}\lim}_{t\to 0} \int_K |u(t,x)-u_0(x)|\; dx =0, &&\textrm{for any compact } K\subset \R^N, \label{u4}
\end{align}
and for all nonnegative $\psi\in \mathcal{D}(\Rdp)$ and arbitrary
$k\in \mathbb{R}\setminus\bigcup_{l\in{\mathbb N}}\,(\alpha_l,\beta_l)$
 there holds
\begin{equation}
\begin{split}
&\int_{\Rdp}|\eta(x,U(g(t,x)))-\eta(x,U(k))| \psi_{t}(t,x)\; dx \; dt\\
&\quad  +\int_{\R_+\times\R^N} (\sgn (g(t,x)-k) ({\cal A}(g(t,x))-{\cal A}(k))) \cdot \nabla \psi(t,x) \; dx\; dt \ge 0.\label{u3_en}
\end{split}
\end{equation}
The numbers $\alpha_l,\beta_l$, $l\in{\mathbb N}$ are defined in \eqref{alphabeta}.

\begin{Rem}[Remark 1.1 from \cite{BuGwSw2013}]
Any entropy weak  solution  is a weak solution to \eqref{E1}-\eqref{E0}.
Indeed, since $g\in  {L^{\infty}(\mathbb{R}_+\times\mathbb{R}^N)}$ we may take $k:=\pm \|g\|_{\infty}$ in \eqref{u3_en} (or possibly we increase/decrease the value of $k$ such that $U$ is strictly increasing in $k$) and by using the strict monotonicity of $\eta$  and the monotonicity of $U$  we conclude that
\begin{align}
&u_{t} + \div A(g) =0, &&\textrm{in the sense of distribution in } \Rdp,\label{u2}
\end{align}
which is exactly \eqref{E1}  {with $f=0$}.
Next, we can use the fact that by functions $|u-\cdot|$ one can generate any convex function and therefore it is a direct consequence of \eqref{u3_en} that (see~\cite{BuGwMaSw2011} for details)  for all smooth convex ${E}$, such that ${E}$ is linear on $(\alpha_k,\beta_k)$ for all $\mathbb{N}$, where $\alpha_k$ and $\beta_k$ are introduced in  \eqref{alphabeta}, there holds
\begin{align}
&Q_u(x,g)_{t} + \div \bQ_{A} \le 0, &&\textrm{in sense of
distribution in }\R_+\times\R^N\label{u3}
\end{align}
with $Q_u$ and $\bQ_{A}$ given by
%\marginpar{\textcolor{blue}{This seems strange since $A$ is only continuous}}
\begin{equation}
\partial_s Q_u(x,s)=\partial_s \eta(x,U(s)){E}'(s), \qquad  \bQ_{A}(s)= \int_0^s{A'(\tau)} {E}'(\tau)\;d\tau.\footnote{Since $A$ is only continuous, then this relation should be understood as follows $\bQ_{A}(s)= A(s)E'(s)-\int_0^s{A(\tau)} {E}''(\tau)\;d\tau.$} 
\label{opet}
\end{equation}
\end{Rem}

Hence from here one easily observes that \eqref{u3_en} does not hold for all $k\in\R$ and the family of entropies is restricted to such that are linear on the intervals $(\alpha_k, \beta_k)$. 
In a consequence we lose the information on the intervals where $\theta$ is multi-valued. In the current paper the situation is significantly different. The approximation of the problem follows in two steps. One is the mollification of the multi-valued term (we take a minimal selection and then mollify with a smooth kernel) and the second one consists in subtracting a strictly monotone perturbation from the source term. Then the right hand side becomes strictly dissipative, namely the inequality in~\eqref{A4} becomes strict for $u\neq v$ and this is the sufficient argument to obtain the uniqueness of entropy measure-valued solutions and to show they reduce to a Dirac measure.  {Here one needs the initial condition.} For passing to the limit with a perturbation of the right-hand side one takes advantage of   the semi-Kru\v zkov entropies $E(u,k)=(u-k)^+$ and $E(u,k)=(u-k)^-$ and then combines the information on the monotonicity of appropriate sequences and boundedness to obtain the strong convergence. Hence this is sufficiently powerful information to provide that on the sets where $\theta$ is multi-valued  
one is not obliged to have linear (or affine) functionals and continuity is enough for the limit passage. 

%In a conseqence one may not use as an entropy the semi-Kru\v zkov entropies
%$E(u,k)=(u-k)^+$ and $E(u,k)=(u-k)^-$. Therefore it would not be possible to use a  comparision principle, which is a powerful tool here. 

We complete this section by referring to other previous results for scalar conservation laws with discontinuous fluxes.  The approach of Panov \cite{Panov} arises from an idea of {\it adapted entropies} introduced for the problems with $x-$discontinuous fluxes in \cite{BaJe1997} and later in \cite{AudussePerthame}. The approach consisted in using in classical Kru\v zkov entropies in place of a constant $k$ the solution to a stationary problem. 
The equivalence between such solutions and entropy weak solutions understood as in \cite{Kr70} in case of smooth fluxes was shown in~\cite{Chen}. There are various different approaches to  fluxes discontinuous in $x$, see e.g. the front tracking method for one dimensional problem, cf.~\cite{Gi1993, KaRiTo2003, Risebro}. The multi-dimensional problem was considered among others in~\cite{AnKaRi2010, Ji2011, Mi2010}. To motivate the studies in the direction of fluxes discontinuous in $u$ we refer to the implicit constitutive theory and the works of Rajagopal, \cite{Ra03}, described also in more detail in \cite{BuGwMaSw2011}.

%WPISAC ODNIESIENIA ZAZNACZONE W PRACY Z MIRKIEM

The paper is organized as follows. In Section~\ref{entropy} we collect all the essential tools needed for the proof of Theorem~\ref{main}.  We start with a contraction principle formulated for entropy measure-valued solutions (Lemma~\ref{contraction}). Then essentially using this result we show a contraction principle for entropy weak solutions  {(Lemma~\ref{22}, estimate~\eqref{begin1})} and comparison principle for entropy 
weak solutions  {(Lemma~\ref{22}, estimate~\eqref{begin2})}. The whole Section~\ref{existence} is dedicated to the proof of Theorem~\ref{main}. We start with regularizing the flux function and then add the strictly monotone perturbation to the source term. The scheme of the proof is first showing the existence of entropy measure-valued solutions, then their uniqueness and finally concluding that the solutions are indeed entropy weak solutions. 
In the final part of the paper there is an appendix which partially recalls the facts from~\cite{BuGwSw2013} and also extends some technical lemmas for the case of multi-valued mappings.  

\section{Entropy inequalities}\label{entropy}
We shall start this section with the definition of entropy measure-valued solutions and then collect the essential estimates used for the proof of existence of solutions: averaged contraction principle and comparison principle. 
\subsection{Averaged contraction principle for entropy measure valued solutions}
We recall that  $\M(\R)$ denotes the space of bounded Radon measures and   $\prob(\R)$  the space of probablity measures, $\mathcal{C}_b(\mathbb{R})$ stands for the space of continuous bounded functions.  {As usual, $\langle\cdot,\cdot\rangle$ denotes the duality pairing between $\mathcal{C}_b(\mathbb{R})$ and $\mathcal{M}(\mathbb{R})$.}
 By a Young measure $\nu$ we mean a weak$^*$ measurable map $\nu:\Rdp\to{\mathcal M}(\R)$ and
  such that $\nu_{(t,x)}\ge0, \|\nu_{(t,x)}\|_{{\mathcal M}(\R)}\le 1$ for a.a. $(t,x)\in\Rdp$. Any bounded sequence  {of measurable functions} $u^n:\Rdp\to\R$ generates a Young measure, which is a probability measure.
  By $L^\infty_w(\Rdp;{\mathcal M}(\R))$ we understand the space of weak$^*$ measurable maps
  $\nu:\Rdp\to{\mathcal M}(\R)$ that are essentially bounded.

%In the current section we will consider a strict monotone perturbation of the source term. Namely, let 
%\begin{equation}
%\tilde f(t,x,r)=f(t,x,r)+\varphi_{m,n}(r)
%\end{equation}
% with
%%$$\beta_{m,n}(r)=\chi_{\{Q\setminus Q\}}m(r-a(t,x))^+-n(a(t,x)-r)^+$$
%$$\varphi_{m,n}(r)=\frac{1}{n}\arctan(r^+)-\frac{1}{m}\arctan(r^-). $$
%Then we consider  problem~\eqref{E1}--\eqref{E0} with $\tilde f$ in the place of $f$. 

%\subsection{Entropy measure-valued solutions}
\begin{definition}\label{DF2}
Let $ \Phi, f$ satisfy the assumptions (H1)--(H4) and $u_0\in L^1_{loc}(\R^N)$. We say that a Young measure $\nu:\Rdp \to
\prob(\R)$ is an  entropy measure-valued solution to
\eqref{E1} if 
there exists $R(t,x)\in L^{\infty}_{loc}(\R_+\times \R^N)$ such that 
\begin{equation}
\begin{split}
\supp \nu_{(t,x)} \subset [-R{(t,x)},R{(t,x)}] \qquad \textrm{ for a.a. } (t,x)\in \Rdp
 \end{split}\label{ApE}
\end{equation}
and if 
for all $\mu \in \R$ and all nonnegative $\psi \in \mathcal{D}(\Rdp)$ there holds

\begin{equation}
\begin{split}
&\int_{\Rdp}\langle (\eta(x,\lambda)-\eta(x,\mu))^+,\nu_{(t,x)}(\lambda)
\rangle \psi_{t}(t,x)\; dx \; dt \\
&\quad + \int_{\Rdp}\langle \chi_{\{\lambda>\mu\}}(A(\lambda)-A(\mu))
 ,\nu_{(t,x)}(\lambda)\rangle \cdot \nabla \psi (t,x) \; dx\; dt \\
&\quad + \int_{\Rdp}
\langle \chi_{\{\lambda>\mu\}} f(t,x,\lambda),\nu_{(t,x)}(\lambda)\rangle\psi \; dx\; dt\ge 0 . \label{EQ+}
\end{split}
\end{equation}
and 
\begin{equation}
\begin{split}
& {-\int_{\Rdp}\langle (\eta(x,\lambda)-\eta(x,\mu))^-,\nu_{(t,x)}(\lambda)
\rangle \psi_{t}(t,x)\; dx \; dt} \\
&\quad + \int_{\Rdp}\langle \chi_{\{\lambda<\mu\}}(A(\lambda)-A(\mu))
 ,\nu_{(t,x)}(\lambda)\rangle \cdot \nabla \psi (t,x) \; dx\; dt \\
&\quad + \int_{\Rdp}
\langle \chi_{\{\lambda<\mu\}} f(t,x,\lambda),\nu_{(t,x)}(\lambda)\rangle \psi\; dx\; dt\le 0 . \label{EQ-}
\end{split}
\end{equation}
Moreover, for all compact $K\subset \R^N$ the following holds
    \begin{equation}\label{ic}
 {\operatorname{ess}\lim}_{t\to 0_+} \int_K \langle |\eta(x,\lambda)-u_0(x)|,\nu_{(t,x)}(\lambda)\rangle \; dx =0.
\end{equation}
%Moreover, there exists $\bF_{\infty}\in L^1_{loc}(\Rdp)$ (behavior as $|x|\to \infty$) and $1\le p \le \frac{d}{d-1}$ such that for any $T>0$ the following holds
%    \begin{equation}
%\begin{split}
%&\int_0^T \left(\int_{\Rd} \left|\langle A(\lambda)-\bF_{\infty}(t,x), \nu_{(t,x)}(\lambda) \rangle \right|^{p}\; dx\right)^{\frac{1}{p}} \; dt < \infty.
% \label{bcc-def}
%\end{split}
%\end{equation}

\end{definition}
%\begin{Rem}
%Note that every monotone bounded function $E'(\lambda)$ can be approximated in the supremum norm by a sequence of 
%functions $E_n'(\lambda)=\sum_{i=1}^n\alpha_i\cdot\chi_{\{\lambda>\mu_i\}}$. 
%Therefore one can replace  in  formulation \eqref{EQ} such   entropy  with general 
%${\mathcal C}^1\cap W^{1,\infty}$ entropies. %as in lemma below.
%\end{Rem}
The existence of entropy measure-valued solutions will be a byproduct of the proof of existence of entropy weak solutions. Below we formulate and prove the estimate (the averaged contraction  principle) which is used both for showing existence and uniqueness of entropy measure-valued solutions.  The proof bases on the method  of doubling the variables, but on the level of measure-valued solutions.
%\subsection{Contraction principle}
\begin{lemma}\label{contraction}
% and let  $(A_1, U_1)$, $(A_2, U_2)$ be two different couples  satisfying \eqref{As1}. 
Assume that $\nu$,  $\sigma$  are two local entropy measure-valued solutions to \eqref{E1}
with a right-hand side $f$ and initial condition $u_0\in L^1_{loc}(\R^N)$. Let\footnote{We keep the general notation $(E,Q)$ instead of writing the concrete form of the entropy and the entropy flux for the sake of the next lemmas and their proofs, where  similar arguments are partially repeated.} $E(\xi)=|\xi|$ with a corresponding flux 
$Q(\lambda,\mu)=\sgn(\lambda-\mu)(A(\lambda)-A(\mu))$. Moreover let
$$E'(\xi):=(\partial E)^0(\xi)=\left\{\begin{array}{rcl}
-1&{\rm for}&\xi<0\\
0&{\rm for}&\xi=0\\
1&{\rm for}&\xi>0 
\end{array}.\right.
$$
%\footnote{We say that $(E, \bQ)$ is an entropy-entropy flux  if $E$ is an arbitrary ${\mathcal C}^1$ function (entropy) and $\bQ$ (flux) satisfies
%$
%\partial_u \bQ(x,u)=E'(u)\partial_u \bF(x,u).
%$}, 
%with $E\in {\cal C}^1(\R)$ and even.
 Then for all nonnegative $\psi \in \mathcal{D}(\Rdp)$ it holds  %\marginpar{zmienic na tilde f}
\begin{equation}
\begin{split}
&\int_{\Rdp}\langle E(\eta(x,\lambda)-\eta(x,\mu)), \nu_{(t,x)}(\lambda)\otimes \sigma_{(t,x)}(\mu)\rangle \psi_{t}(t,x)\; dx\; dt\\
&\quad + \int_{\Rdp}\langle Q(\lambda, \mu), \nu_{(t,x)}(\lambda)\otimes \sigma_{(t,x)}(\mu)\rangle \cdot \nabla \psi(t,x)\; dx \; dt \\
&\quad +\int_{\Rdp}\langle E'(\lambda-\mu)(f(t,x,\lambda)-f(t,x,\mu)), \nu_{(t,x)}(\lambda)\otimes \sigma_{(t,x)}(\mu)\rangle\psi(t,x)\;dx\;dt
\ge 0 \label{begin3}
\end{split}
\end{equation}

\end{lemma}
\begin{proof}%[Proof of Lemma~\ref{contraction}]
%Assume that $E\in {\cal C}^1(\R)$. 
Let $\omega\in \mathcal{D}(-1,1)$ be a regularizing kernel, i.e.,
$\omega(x)=\omega(-x)$ and $\int_{-1}^1 \omega(x)\; dx =1$. Then, for any $\gamma>0$, we define
\begin{align*}
\omega_1^{\gamma}(t)&:=\gamma^{-1} \omega (t/\gamma) &&\textrm{for all } t\in \R,\\
\omega_2^{\gamma}(x)&:= \gamma^{-N} \omega(x_1/\gamma)\cdot \ldots\cdot \omega(x_N/\gamma) &&\textrm{for all } x=(x_1,\ldots, x_N)\in \mathbb{R}^N.
\end{align*}
For arbitrary  $\varepsilon,\delta>0$ we set $\omega^{\delta,\varepsilon}(t,x):=\omega_1^{\delta}(t)\cdot \omega_2^{\varepsilon}(x)$. Notice that for any Young measure
$\nu \in L^{\infty}_{w}([0,T]\times \R^N; \mathcal{M}(\R))$ there exists a Young measure
$\nu^\delta\in  L^{\infty}_{w}(\R^N;\mathcal{C}^{\infty}([0,T]; \mathcal{M} (\R))) $ with
$\|\nu^\delta\|_{L^{\infty}_{w}([0,T]\times \R^N; \mathcal{M}(\R))}\le 1$
such that for any  $f\in  {\mathcal{C}_b(\R)}$ the following holds\footnote{We extend the measure for $t<0$ and $t>T$ by zero.} $(\omega_1^\delta*\langle  f, \nu \rangle)=\langle f, \nu^{\delta}
\rangle$ for almost all $t\in \R$. Moreover, we can interchange the derivative as
$\langle f, \partial_t\nu^\delta \rangle=\langle f, \nu^\delta \rangle_{t}$ for all $t\in \R$. Similarly, there exists $\nu^\varepsilon \in L^\infty_w([0,T];\mathcal{C}^{\infty}(\R^N_{loc}; \mathcal{M} (\R)))$ with
$\|\nu^\varepsilon\|_{L^{\infty}_{w}([0,T]\times \R^N; \mathcal{M}(\R))}\le 1$ such that $\omega_2^\varepsilon *\langle f, \nu \rangle
=\langle f, \nu^\varepsilon\rangle$ and $\langle  f, \partial_{x_i} \nu^\varepsilon\rangle=\partial_{x_i}\langle  f, \nu^\varepsilon \rangle$ for all $x\in \R^N$, see  Ref.~\cite{DiPerna}.

Let $Q(\lambda,\mu):= \sgn(\lambda-\mu)(A(\lambda)-A(\mu))$. Then
$$(\nu,\sigma)\mapsto \langle Q(\lambda, \mu),\nu\otimes \sigma\rangle \in\R$$
is a  bounded bilinear form from $\M(\R)\times \M(\R)$ to $\R$ and 
$$(t,x)\mapsto \nu_{(t,x)}^{\varepsilon,\delta}\in {\cal C}^\infty (K,(\M(\R), \|\cdot\|_{\M}))$$
$$(t,x)\mapsto \sigma_{(t,x)}^{\varepsilon,\delta}\in {\cal C}^\infty (K,(\M(\R), \|\cdot\|_{\M}))$$
for any compact $K\subset \R_+\times \R^N$  and then 
\begin{equation}\label{div}
\div\langle Q(\lambda, \mu), \nu_{(t,x)}^\varepsilon\otimes\sigma_{(t,x)}^\varepsilon\rangle
=\langle Q(\lambda, \mu), \nabla \nu_{(t,x)}^\varepsilon\otimes\sigma_{(t,x)}^\varepsilon\rangle
+\langle Q(\lambda, \mu), \nu_{(t,x)}^\varepsilon\otimes\nabla \sigma_{(t,x)}^\varepsilon\rangle.
\end{equation}
%We proceed formally, however the whole proof can be done rigorously by approximating w.r.t. $x$ and $t$ 
%the terms $\langle |\eta(x,\lambda)-\eta(x,\mu)|, \nu_{(t,x)}(\lambda)\rangle$ and 
%$\langle \sgn(\lambda-\mu)(A(x,\lambda)-A(x,\mu)), \nu_{(t,x)}(\lambda)\rangle$ (analogously for $\sigma$).
%Note that, since $A$ depends on $x$, hence we cannot, following~\cite{DiPerna}, search for a measure 
%$\nu^\varepsilon\in {\cal C}^\infty(0,T;\prob\R)$ such that 
%for any  $f\in \mathcal{C}(\R)$ the following  would hold $(\omega_1^\varepsilon*\langle  f, \nu \rangle)=\langle f, \nu^{\varepsilon}
%\rangle$ for almost all $t\in \R$ and the derivatives could be  interchanged  as follows
% $\langle  f, \partial_{x_i} \nu^\varepsilon\rangle=\partial_{x_i}\langle  f, \nu^\varepsilon \rangle$ for all $x\in \R^N$. Nevertheless, this is not necessary.
For arbitrary nonnegative $\psi\in \mathcal{D}(\Rdp)$ we observe that
for all $\mu \in \R$
\begin{equation}
\begin{split}
&\int_{\Rdp} \langle E(\eta(x,\lambda)-\eta(x,\mu)),\nu_{(t,x)}(\lambda)
\rangle\  (\psi * (\omega_1^{\delta}\cdot \omega_2^{\varepsilon}))_{t} \; dx\; dt\\
%&=\int_{\Rdp} \langle |\eta(x,U_1(\lambda))-\eta(x,U_1(\mu))|,\nu^{\varepsilon}(\lambda)
%\rangle\   (\psi * (\omega_2^{\varepsilon} ) )_{,t}\; dx\; dt\\
&=\int_{\Rdp}\omega_2^{\varepsilon}*\langle E(\eta(x,\lambda)-\eta(x,\mu)),\nu_{(t,x)}^{\delta}(\lambda)
\rangle\,  \psi_{t}  \; dx\; dt.
\end{split}\label{3.7}
\end{equation}
Similarly, we obtain for all $\mu \in \R$
\begin{equation}
\begin{split}
&\int_{\Rdp}\langle Q(\lambda,\mu),\nu_{(t,x)}(\lambda)\rangle \cdot \nabla (\psi * (\omega_1^{\delta}\cdot \omega_2^{\varepsilon}))\; dx\; dt=\\
%&=\int_{\Rdp}\langle (\bA_1(\lambda)-\bA_1(\mu))
%\sgn (\lambda-\mu),\nu^{\varepsilon}(\lambda)\rangle \cdot \nabla (\psi * (\omega_1^{\varepsilon}))\; dx\; dt=\\
&=\int_{\Rdp} \langle Q(\lambda,\mu),\nu_{(t,x)}^{\delta,\varepsilon}(\lambda)\rangle \cdot \nabla \psi \; dx\; dt.
%+\int_{\Rdp}{\mathcal R}_\varepsilon \cdot \nabla  \psi\;dx\;dt \label{3.8}
\end{split}
\end{equation}
Moreover
\begin{equation}\label{2.11}\begin{split}
&\int_{\Rdp}\langle E'(\lambda-\mu)f(t,x,\lambda), \nu_{(t,x)}(\lambda)\rangle \psi * (\omega_1^{\delta}\cdot \omega_2^{\varepsilon}) \;dx\;dt\\
&=\int_{\Rdp}(\omega_1^{\delta}\cdot \omega_2^{\varepsilon})*\langle E'(\lambda-\mu)f(t,x,\lambda), \nu_{(t,x)}(\lambda)\rangle \psi \;dx\;dt.
\end{split}\end{equation}
%and 
%\begin{equation}\begin{split}
%&\int_{\Rdp}\langle E'(\lambda-\mu)f_2(t,x,\lambda), \nu_{(t,x)}(\lambda)\rangle \psi * (\omega_1^{\delta}\cdot \omega_2^{\varepsilon}) \;dx\;dt\\
%&=\int_{\Rdp}(\omega_1^{\delta}\cdot \omega_2^{\varepsilon})*\langle E'(\lambda-\mu)f_2(t,x,\lambda), \nu_{(t,x)}(\lambda)\rangle \psi \;dx\;dt
%\end{split}\end{equation}
Summing \eqref{EQ+} and \eqref{EQ-} we obtain an entropy inequality with the entropy $E(\xi)=|\xi|$, where we may take $\psi * (\omega_1^{\delta}\cdot \omega_2^{\varepsilon})$ as a test function  
and  using \eqref{3.7}--\eqref{2.11}  we deduce that
for all $\mu \in \R$ and all nonnegative $\psi \in \mathcal{D}((\varepsilon,\infty)\times \R^N)$ there holds
\begin{equation}
\begin{split}
&\int_{\Rdp}\omega_2^{\varepsilon}*\langle E(\eta(x,\lambda)-\eta(x,\mu)),\nu_{(t,x)}^{\delta}(\lambda)
\rangle \psi_{t}\; dx \; dt \\
&\quad +\int_{\Rdp} \langle Q(\lambda,\mu),\nu_{(t,x)}^{\delta,\varepsilon}(\lambda)\rangle \cdot \nabla \psi \; dx\; dt
%+\int_{\Rdp}{\mathcal R}_\varepsilon \cdot \nabla  \psi\;dx\;dt 
\\
&\quad+\int_{\Rdp}(\omega_1^{\delta}\cdot \omega_2^{\varepsilon})*\langle E'(\lambda-\mu)
f(t,x,\lambda), \nu_{(t,x)}(\lambda)\rangle \psi \;dx\;dt \ge0\label{regularized}
\end{split}
\end{equation}
which in particular implies that for all $\tilde{\mu}\in \R$ and all $(t,x)\in (\varepsilon,\infty)\times \R^N$ there holds
\begin{equation}
\begin{split}
\left(\omega_2^{\varepsilon}*\langle E(\eta(x,\lambda)-\eta(x,\tilde \mu)),\nu_{(t,x)}^\delta(\lambda)
\rangle \right)_{t} 
 + \div  \langle Q(\lambda,\tilde \mu),\nu_{(t,x)}^{\delta,\varepsilon}(\lambda)\rangle 
%+\div {\cal R}_\varepsilon 
\\
 \le (\omega_1^{\delta}\cdot \omega_2^{\varepsilon})*\langle E'(\lambda-\tilde \mu)
f(t,x,\lambda), \nu_{(t,x)}(\lambda)\rangle. \label{reg2}
\end{split}
\end{equation}
Similarly we have
%, for a Young measure $\sigma$ 
%and functions $U_2$ and  $A_2$, 
%we can deduce that
 for any 
$\varepsilon>0$,   $\tilde \lambda \in \R$ and all  $(t,x)\in (\varepsilon,\infty)\times \R^N$ %we have
\begin{equation}
\begin{split}
\left(\omega_2^{\varepsilon}*\langle E(\eta(x,\tilde \lambda)-\eta(x, \mu)),\sigma_{(t,x)}^\delta(\mu)
\rangle \right)_{t} 
 + \div  \langle Q(\tilde \lambda, \mu),\sigma_{(t,x)}^{\delta,\varepsilon}(\mu)\rangle 
%+\div {\cal R}_\varepsilon
 \\
 \le (\omega_1^{\delta}\cdot \omega_2^{\varepsilon})*\langle E'( \mu-\tilde \lambda)
f(t,x,\mu), \sigma_{(t,x)}(\mu)\rangle. \label{reg3}
\end{split}
\end{equation}

We apply $\sigma_{(t,x)}^{\delta,\varepsilon}$ to \eqref{reg2}. Note that the left-hand side  is a continuous function of $\mu$ and the right-hand side is only a Borel function of $\mu$.  Similarly we apply $\nu_{(t,x)}^{\delta,\varepsilon}$ onto \eqref{reg3}. Summing the resulting expressions %and using \eqref{use1}--\eqref%{use2} 
we find that for all $(t,x)\in (2\varepsilon,\infty)\times \R^N$ there holds
\begin{equation}
\begin{split}
\langle \omega_2^{\varepsilon}*&\langle E(\eta(x,\lambda)-\eta(x,\mu)), \nu_{(t,x)}^\delta(\lambda)\rangle_{t}, \sigma_{(t,x)}^{\delta,\varepsilon}(\mu)\rangle\\
&+\langle \omega_2^{\varepsilon}*\langle E(\eta(x,\lambda)-\eta(x,\mu)), \sigma_{(t,x)}^\delta(\mu)\rangle_{t}, \nu_{(t,x)}^{\delta,\varepsilon}(\lambda)\rangle\\
&+\div \langle  Q(\lambda,\mu), \nu_{(t,x)}^{\delta,\varepsilon}(\lambda)\otimes
 \sigma_{(t,x)}^{\delta,\varepsilon}(\mu)\rangle %+\div \langle  {\cal R}_\varepsilon, \nu_{(t,x)}^{\delta,\varepsilon}(\lambda)\otimes
 %\sigma_{(t,x)}^{\delta,\varepsilon}(\mu)\rangle
 \\
&\le \langle (\omega_1^{\delta}\cdot \omega_2^{\varepsilon})*\langle E'(\lambda- \mu)
f(t,x,\lambda), \nu_{(t,x)}(\lambda)\rangle,   \sigma^{\delta,\varepsilon}_{(t,x)}(\mu)\rangle
\\
&+
\langle (\omega_1^{\delta}\cdot \omega_2^{\varepsilon})*\langle E'(\mu-\lambda)
f(t,x,\mu), \sigma_{(t,x)}(\mu)\rangle,   \nu^{\delta,\varepsilon}_{(t,x)}(\lambda)\rangle=:I
%\\
%&+(\omega_1^{\delta}\cdot \omega_2^{\varepsilon})*\langle E'(\tilde \lambda- \mu)
%(f_1(t,x, \lambda)-f_2(t,x, \mu), \sigma_{(t,x)}(\mu)\rangle
\end{split}
\label{regf}
\end{equation}
To  proceed with a righ-hand side we define the errors as follows
\begin{equation}\begin{split}
{\cal R}_{\varepsilon,\delta,n}^\lambda:=&
\langle(\omega_1^{\delta}\cdot \omega_2^{\varepsilon})*\langle E'(\lambda- \mu)
(f(t,x,\lambda)-f^n(t,x,\lambda), \nu_{(t,x)}(\lambda)\rangle,   \sigma^{\delta,\varepsilon}_{(t,x)}(\mu)\rangle\\
+&
\langle(\omega_1^{\delta}\cdot \omega_2^{\varepsilon})*\langle E'(\lambda- \mu)
f^n(t,x,\lambda), \nu_{(t,x)}(\lambda)\rangle,   \sigma^{\delta,\varepsilon}_{(t,x)}(\mu)\rangle\\
-&\langle\langle E'(\lambda- \mu)
f^n(t,x,\lambda), \nu^{\delta,\varepsilon}_{(t,x)}(\lambda)\rangle,   \sigma^{\delta,\varepsilon}_{(t,x)}(\mu)\rangle
\end{split}\end{equation}
and 
\begin{equation}\begin{split}
{\cal R}_{\varepsilon,\delta,n}^\mu:=&
\langle(\omega_1^{\delta}\cdot \omega_2^{\varepsilon})*\langle E'(\mu-\lambda)
(f(t,x,\mu)-f^n(t,x,\mu), \sigma_{(t,x)}(\mu)\rangle, \nu^{\delta,\varepsilon}_{(t,x)}(\lambda)  \rangle\\
+&
\langle(\omega_1^{\delta}\cdot \omega_2^{\varepsilon})*\langle E'(\mu-\lambda)
f^n(t,x,\mu), \sigma_{(t,x)}(\mu)\rangle,   \nu^{\delta,\varepsilon}_{(t,x)}(\lambda)\rangle\\
-&\langle\langle E'(\mu-\lambda)
f^n(t,x,\mu), \sigma^{\delta,\varepsilon}_{(t,x)}(\mu)\rangle,   \nu^{\delta,\varepsilon}_{(t,x)}(\lambda)\rangle
\end{split}\end{equation}
where $(f^n)_{n\in{\mathbb N}}$ is the sequence of uniformly continuous functions in $(t,x)$ and continuous in $u$
 % having a modulus of continuity
   and there exists  an $L_K(n)$ such that for a fixed compact $K$ it vanishes as $n\to\infty$ and 
\begin{equation}
\sup\limits_{\lambda\in K}\|f(\cdot,\cdot,\lambda)-f^n(\cdot,\cdot,\lambda)\|_{L^1(\Rdp)}\le L_K(n).
\end{equation}
Let  ${\cal W}_K^n$ be a modulus of continuity of the function $f^n$, namely  ${\cal W}_K^n:\R^2_+\to\R_+$ is  continuous,  ${\cal W}_K^n(0,0)=0$ and 
\begin{equation}
\sup\limits_{\lambda\in K}|f^n(t-s,x-y,\lambda)-f^n(t,x,\lambda)|\le {\cal W}_K^n(|s|,|y|)
\end{equation}
where $K$ is an arbitrary  compact subset of $\R$.
Hence 
\begin{equation}
\begin{split}
I&=\langle\langle E'(\lambda- \mu)
f^n(t,x,\lambda), \nu^{\delta,\varepsilon}_{(t,x)}(\lambda)\rangle,   \sigma^{\delta,\varepsilon}_{(t,x)}(\mu)\rangle +{\cal R}_{\varepsilon,\delta,n}^\lambda\\
&+\langle\langle E'(\mu-\lambda)
f^n(t,x,\mu), \sigma^{\delta,\varepsilon}_{(t,x)}(\mu)\rangle,   \nu^{\delta,\varepsilon}_{(t,x)}(\lambda)\rangle
+{\cal R}_{\varepsilon,\delta,n}^\mu
\end{split}\end{equation}
and as  $E'(\xi)=-E'(-\xi)$ and  using the Fubini theorem we further conclude
\begin{equation}
I=\langle\langle E'(\lambda- \mu)
(f^n(t,x,\lambda)-f^n(t,x,\mu)), \nu^{\delta,\varepsilon}_{(t,x)}(\lambda)\rangle,   \sigma^{\delta,\varepsilon}_{(t,x)}(\mu)\rangle +{\cal R}_{\varepsilon,\delta,n}^\lambda+{\cal R}_{\varepsilon,\delta,n}^\mu. 
\end{equation}
Note that the function $E'(\lambda- \mu)
(f^n(t,x,\lambda)-f^n(t,x,\mu))$ is continuous, although $E'(\lambda- \mu)=\sgn(\lambda-\mu)$ is not continuous
for $\lambda-\mu=0$.
We shall estimate the error ${\cal R}_{\varepsilon,\delta,n}^\lambda$, the estimates for ${\cal R}_{\varepsilon,\delta,n}^\mu$ follow the same lines. 
%Observe that in the terms on the right-hand side the convolution in the dual pair can be transferred to a measure 
%%$\sigma^{\delta,\varepsilon}$ and  
%$\nu^{\delta,\varepsilon}$.
%% respectively. 
%For this purpose we define
%%$$\tilde\sigma^{\delta,\varepsilon}:=(\omega_1^{\delta}\cdot \omega_2^{\varepsilon})*\sigma^{\delta,\varepsilon}$$
%%and 
%$$\tilde\nu^{\delta,\varepsilon}:=(\omega_1^{\delta}\cdot \omega_2^{\varepsilon})*\nu^{\delta,\varepsilon}.$$
Then
\begin{equation}\label{error-es}
\begin{split}
|{\cal R}_{\varepsilon,\delta,n}^\lambda|&\le
|\langle(\omega_1^{\delta}\cdot \omega_2^{\varepsilon})*\langle E'(\lambda- \mu)
(f(t,x,\lambda)-f^n(t,x,\lambda), \nu_{(t,x)}(\lambda)\rangle,   \sigma^{\delta,\varepsilon}_{(t,x)}(\mu)\rangle|\\
&
+|\int_\R\int_{\R\times \R^N}\omega_1^\delta(s)\omega_2^\varepsilon(y)\langle E'(\lambda-\mu)
f^n(t-s,x-y,\lambda),\nu_{(t-s,x-y)}(\lambda)\rangle \; dy\; ds \; d\sigma^{\delta,\varepsilon}_{(t,x)}(\mu)\\
&-\int_\R\langle E'(\lambda-\mu) f^n(t,x,\lambda),\nu^{\delta,\varepsilon}_{(t,x)}(\lambda)\rangle \; d\sigma^{\delta,\varepsilon}_{(t,x)}(\mu) |\\
&\le\sup\limits_{\lambda\in K}\|f(\cdot,\cdot,\lambda)-f^n(\cdot,\cdot,\lambda)\|_{L^1(\Rdp)}\\&+
 \sup\limits_{|t-s|\le\delta,\,|x-y|\le\varepsilon}|f^n(t-s,x-y,\lambda)-f^n(t,x,\lambda)| 
%\;d\sigma^{\delta,\varepsilon}_{(t,x)}(\mu)
\\
&\le L_K(n)+ {\cal W}_K^n(\delta,\varepsilon).
\end{split}\end{equation}

Thus, multiplying \eqref{regf} by an arbitrary fixed nonnegative $\psi \in \mathcal{D}((2\varepsilon,\infty )\times \R^N)$, integrating the result over $\Rdp$ and using integration by parts, we find that
\begin{equation}
\begin{split}
-\int_{\Rdp} &\left(\left\langle \omega_2^{\varepsilon}*\langle |\eta(x,\lambda)-\eta(x,\mu)|, \nu_{(t,x)}^\delta(\lambda)\rangle_{t}, \sigma_{(t,x)}^{\delta,\varepsilon}(\mu)\right\rangle\right.\\&+
\left.\left\langle \omega_2^{\varepsilon}*\langle |\eta(x,\lambda)-\eta(x,\mu)|, \sigma_{(t,x)}^\delta(\mu)\rangle_{t}, \nu_{(t,x)}^{\delta,\varepsilon}(\lambda)\right\rangle\right)\psi\; dx\; dt \\
& +\int_{\Rdp}\left\langle  Q(\lambda,\mu), \nu_{(t,x)}^{\delta,\varepsilon}(\lambda)\otimes \sigma_{(t,x)}^{\delta,\varepsilon}(\mu)\right\rangle
\cdot \nabla \psi \; dx \; dt \\
%&+
%\int_{\Rdp}\left\langle  {\cal R}_\varepsilon(t,x,\mu), \nu_{(t,x)}^{\delta,\varepsilon}(\lambda)\otimes \sigma_{(t,x)}^{\delta,\varepsilon}(\mu)\right\rangle
%\cdot \nabla \psi \; dx \; dt\\
&\ge -
\int_{\Rdp} %\left(
  \langle E'(\lambda- \mu)
(f^n(t,x,\lambda)-f^n(t,x, \mu)), \nu^{\delta,\varepsilon}_{(t,x)}(\lambda)\otimes  \sigma^{\delta,\varepsilon}_{(t,x)}(\mu)\rangle
%&\left.+
% \langle E'(\lambda- \mu)
%(f_1(t,x,\lambda)-f_2(t,x, \mu), \sigma_{(t,x)}(\mu)\otimes \tilde \nu^{\delta,\varepsilon}_{(t,x)}(\lambda)\rangle
%\right)
%
%
%
%2 (\omega_1^{\delta}\cdot \omega_2^{\varepsilon})*\langle E'(\lambda- \mu)
%(f_1(t,x,\lambda)-f_2(t,x, \mu), \nu_{(t,x)}(\lambda)\otimes  \sigma_{(t,x)}(\mu)\rangle
\psi \; dx \; dt\\
&-2\int_{\Rdp}({\cal W}_K^n(\delta,\varepsilon)+L_K(n))  \psi \; dx \; dt
\end{split}\label{regf0}
\end{equation}

First, we let $\varepsilon\to0_+$. Then let $\Omega_\psi:=\supp \psi$.
From \eqref{ApE} it follows that there exists a compact set $K$ such that for $(t,x)\in\Omega_\psi$ we have
  $\supp\nu^\delta_{(t,x)}\subset K$ and then also $\supp\partial_t\nu^\delta_{(t,x)}\subset K$.
  The same holds for $\sigma^\delta_{(t,x)}$.
  
Since $\theta$ is bounded by some function independent of $x$, then there exists a function $h_3$, 
again independent of $x$, such that for all $x\in\R^N$ and all $v\in\R$
$$|\eta(x,v)|\le h_3(v),$$
which  provides that  $\eta\in L^\infty(\Rdp;{\mathcal C}(K))$, where $ {(t,x)\mapsto \eta(t,x,\cdot)}$, hence also
  $\eta\in L^1(\Omega_\psi;{\mathcal C}(K))$. Consequently  $E(\eta(x,\lambda)-\eta(x,\mu))\in L^1(\Omega_\psi;{\mathcal C}(K))$.
Thus we can extract a subsequence, that we do not relabel, such that
%for any compact set $K\subset\R$
\begin{align*}
% {\mu\mapsto}
\omega_2^{\varepsilon}*\langle E(\eta(\cdot,\lambda)-\eta(\cdot,\mu)), \partial_t\nu^\delta\rangle&\to\langle 
E(\eta(\cdot,\lambda)-\eta(\cdot,\mu)), \partial_t\nu^\delta\rangle
&&\mbox{ strongly in }&& L^1(\Omega_\psi;\mathcal{C}(K)),\\
% {\lambda\mapsto}
\omega_2^{\varepsilon}*\langle E(\eta(\cdot,\lambda)-\eta(\cdot,\mu)), \partial_t\sigma^\delta\rangle&\to\langle E(\eta(\cdot,\lambda)-\eta(\cdot,\mu)), \partial_t\sigma^\delta\rangle
&&\mbox{ strongly in }&&L^1(\Omega_\psi;\mathcal{C}(K)),\\
\sigma^{\delta,\varepsilon}&\rightharpoonup^*\sigma^{\delta}&&\mbox{ weakly$^*$ in }&&
L^\infty_{w}(\Omega_\psi;{\mathcal M}(K)),\\
\nu^{\delta,\varepsilon}&\rightharpoonup^*\nu^{\delta}&&\mbox{ weakly$^*$ in }&&
L^\infty_{w}(\Omega_\psi;{\mathcal M}(K)),\\
%\omega_K(\delta,\varepsilon)\to \omega_K(\delta,0)
%\tilde\sigma^{\delta,\varepsilon}&\rightharpoonup^*\tilde\sigma^{\delta}&&\mbox{ weakly$^*$ in }&&
%L^\infty_{w}(\Omega_\psi;{\mathcal M}(K)),\\
%\tilde\nu^{\delta,\varepsilon}&\rightharpoonup^*\tilde\nu^{\delta}&&\mbox{ weakly$^*$ in }&&
%L^\infty_{w}(\Omega_\psi;{\mathcal M}(K)).
\end{align*}
%where 
%%$$\tilde\sigma^{\delta}:=\omega_1^{\delta}*\sigma^{\delta}$$
%%and 
%$$\tilde\nu^{\delta}:=\omega_1^{\delta}*\nu^{\delta}.$$
as $\varepsilon\to0$.
Using these convergence results, we observe from \eqref{regf0} that
\begin{equation}\label{regf2}
\begin{split}
&-\int_{\Rdp}  \langle \langle E(\eta(x,\lambda)-\eta(x,\mu)), \nu_{(t,x)}^{\delta}(\lambda)\rangle_{t},
 \sigma_{(t,x)}^{\delta}(\mu)\rangle \psi\; dx\; dt\\
 &\qquad-\int_{\Rdp}
\langle\langle E(\eta(x,\lambda)-\eta(x,\mu)), \sigma_{(t,x)}^{\delta}(\mu)\rangle_{t},
\nu_{(t,x)}^{\delta}(\lambda)\rangle\psi\; dx\; dt \\
 &\qquad +\int_{\Rdp}\langle  \bQ(\lambda,\mu), \nu_{(t,x)}^{\delta}(\lambda)\otimes
  \sigma_{(t,x)}^{\delta}(\mu)\rangle
\cdot \nabla \psi \; dx \; dt \\ %\ge  0.
&\qquad\ge -
\int_{\Rdp}
  \langle E'(\lambda- \mu)
(f^n(t,x,\lambda)-f^n(t,x, \mu)), \nu^\delta_{(t,x)}(\lambda)\otimes   \sigma^{\delta}_{(t,x)}(\mu)\rangle
%\\
%&\qquad\left.+
% \langle E'(\lambda- \mu)
%(f_1(t,x,\lambda)-f_2(t,x, \mu), \sigma_{(t,x)}(\mu)\otimes \tilde \nu^{\delta}_{(t,x)}(\lambda)\rangle
%\right)
%
%
%
%2 (\omega_1^{\delta}\cdot \omega_2^{\varepsilon})*\langle E'(\lambda- \mu)
%(f_1(t,x,\lambda)-f_2(t,x, \mu), \nu_{(t,x)}(\lambda)\otimes  \sigma_{(t,x)}(\mu)\rangle
\psi \; dx \; dt\\
&\qquad- 2
\int_{\Rdp} ({\cal W}_K^n(\delta, 0)+L_K(n))
  \psi \; dx \; dt.
\end{split}\end{equation}
Similarly to~\eqref{div} it is not difficult to observe that
\begin{equation}
\begin{split}
  \langle  E(\eta(x,\lambda)-\eta(x,\mu)), \nu_{(t,x)}^\delta\otimes \sigma_{(t,x)}^\delta\rangle_{t}&=\langle \langle E(\eta(x,\lambda)-\eta(x,\mu)), \nu_{(t,x)}^{\delta}\rangle, \sigma_{(t,x)}^{\delta}\rangle_{t}\\
&=\langle \omega^{\delta}*\langle E(\eta(x,\lambda)-\eta(x,\mu)), \nu_{(t,x)}\rangle, \sigma_{(t,x)}^\delta\rangle_{t}
\\
%&=\langle \omega^{\delta}*\langle \zeta, (\nu^{\delta})_{,t}\rangle, \sigma^{\delta}\rangle+
%\langle \omega^{\delta}*\langle \zeta,  \nu^{\delta}\rangle, (\sigma^{\delta})_{,t}\rangle\\
&=\left\langle ( \omega^{\delta}*\langle \zeta,  \nu_{(t,x)} \rangle)_{t}, \sigma_{(t,x)}^\delta\right\rangle+
\left \langle (\omega^{\delta}*\langle \zeta,  \sigma_{(t,x)}
\rangle)_{t},\nu_{(t,x)}^\delta \right\rangle.
\end{split}\label{chain}
\end{equation}
Thus, using  \eqref{regf2}, \eqref{chain} and integrating by parts with respect to $t$, we find that
\begin{equation*}
\begin{split}
&\int_{\Rdp}  \langle E(\eta(x,\lambda)-\eta(x,\mu)), \nu_{(t,x)}^\delta(\lambda)\otimes
\sigma_{(t,x)}^\delta(\mu)\rangle\psi_{t}\; dx\; dt \\
 &\qquad +\int_{\Rdp}\langle  \bQ(\lambda,\mu), \nu_{(t,x)}^\delta(\lambda)\otimes
 \sigma_{(t,x)}^\delta(\mu) \rangle
\cdot \nabla \psi \; dx \; dt   \\
&\qquad\ge-\int_{\Rdp}
  \langle E'(\lambda- \mu)
(f^n(t,x,\lambda)-f^n(t,x, \mu)), \nu^\delta_{(t,x)}(\lambda)\otimes   \sigma^\delta_{(t,x)}(\mu)\rangle
\\
&\qquad- 2
\int_{\Rdp} ({\cal W}_K^n(\delta,0)+L_K(n))
  \psi \; dx \; dt.
\end{split}
\end{equation*}
Letting $\delta \to 0_{+}$ we conclude  by the argument of weak$^*$  convergence of  measures $ \nu^\delta$ and $\sigma^\delta$ to $\nu$ and $\sigma$, respectively and 
${\cal W}_K^n(\delta,0)\to0$.
\begin{equation*}
\begin{split}
&\int_{\Rdp}  \langle E(\eta(x,\lambda)-\eta(x,\mu)), \nu_{(t,x)}(\lambda)\otimes
\sigma_{(t,x)}(\mu)\rangle\psi_{t}\; dx\; dt \\
 &\qquad +\int_{\Rdp}\langle  \bQ(\lambda,\mu), \nu_{(t,x)}(\lambda)\otimes
 \sigma_{(t,x)}(\mu) \rangle
\cdot \nabla \psi \; dx \; dt   \\
&\qquad\ge-\int_{\Rdp}
  \langle E'(\lambda- \mu)
(f^n(t,x,\lambda)-f^n(t,x, \mu)), \nu_{(t,x)}(\lambda)\otimes   \sigma_{(t,x)}(\mu)\rangle  \psi \; dx \; dt
\\
&\qquad-2
\int_{\Rdp} L_K(n)
  \psi \; dx \; dt.
\end{split}
\end{equation*}

In the final step we let $n\to\infty$ and since $f^n\to f$ in $L^1(\Omega_\psi;{\cal C}(K))$ we  obtain \eqref{begin3}.

\end{proof}

\subsection{Comparison and contraction principles  for entropy  weak solutions}
In the next lemma we included contraction and comparison principle for entropy weak solutions. In order not to 
involve the method of doubling the variables for weak solutions we use as much as possible the results obtained for measure-valued solutions. Here we consider the solutions $v_1$ and $v_2$ corresponding to the problems with different right-hand side. The purpose is to work later with  approximated problems, where the source term shall be perturbed with a strictly monotone term and for the sake of constructing monotone families of approximated sequence we shall be interested in different parameters. 
\begin{lemma}\label{22}
Assume that $v_1$,  $v_2$  are two entropy weak solutions to \eqref{E1}
with a right-hand side $f_1$ and $f_2$ respectively. 
%Moreover, let $(E, \bQ)$ be an entropy-entropy flux pair.
%\footnote{We say that $(E, \bQ)$ is an entropy-entropy flux  if $E$ is an arbitrary convex ${\mathcal C}^1$ function (entropy) and $\bQ$ (flux) satisfies
%$
%\partial_u \bQ(u)=E'(u)\partial_u \bF(u).
%$}
%\footnote{We say that $(E, \bQ)$ is an entropy-entropy flux  if $E$ is an arbitrary ${\mathcal C}^1$ function (entropy) and $\bQ$ (flux) satisfies
%$
%\partial_u \bQ(x,u)=E'(u)\partial_u \bF(x,u).
%$}, 
%with $E\in{\mathcal C}^1(\R)$ and even. 
%with  a Lipschitz continuous function $E$ and even.
 Then %\marginpar{zmienic na tilde f}
 \begin{enumerate}
\item for all nonnegative $ {\psi} \in \mathcal{D}(\Rdp)$ it holds
\begin{equation}
\begin{split}
&\int_{\Rdp} |\eta(x,v_1)-\eta(x,v_2)| \psi_{t}(t,x)\; dx\; dt\\
&\quad + \int_{\Rdp} \sgn(v_1-v_2)(A(v_1)-A(v_2)) \cdot \nabla \psi(t,x)\; dx \; dt \\
&\quad +\int_{\Rdp} \sgn(v_1-v_2)(f_1(t,x,v_1)-f_2(t,x,v_2))\psi(t,x)\;dx\;dt\\
&\quad \ge-\int_{\{(t,x):v_1=v_2\}} |f_1(t,x,v_1)-f_2(t,x,v_2)|\psi(t,x)\;dx\;dt
\label{begin1}
\end{split}
\end{equation}
\item for all nonnegative $ {\psi} \in \mathcal{D}(\Rdp)$ it holds
\begin{equation}
\begin{split}
&\int_{\Rdp} (\eta(x,v_1)-\eta(x,v_2))^+ \psi_{t}(t,x)\; dx\; dt\\
&\quad + \int_{\Rdp} \chi_{\{v_1>v_2\}}(A(v_1)-A(v_2)) \cdot \nabla \psi(t,x)\; dx \; dt \\
&\quad +\int_{\Rdp} \chi_{\{v_1>v_2\}}(f_1(t,x,v_1)-f_2(t,x,v_2))\psi(t,x)\;dx\;dt\\
&\quad \ge -\int_{\{(t,x):v_1=v_2\}} (f_1(t,x,v_1)-f_2(t,x,v_2))^+\psi(t,x)\;dx\;dt 
\label{begin2}
\end{split}
\end{equation}

\end{enumerate}
%Once we know that  entropy measure valued solutions are  entropy weak solution, namely for a.a. 
%$t,x$ there exist $u,v\in L^\infty$ such that $\nu_{(t,x)}=\delta_{u(t,x)}$ and $\sigma_{(t,x)}=\delta_{v(t,x)}$, 
%then from \eqref{begin} with entropy $E_1$ from Remark~\ref{rem1} we conclude the Kato inequality 
%%\eqref{Kato-eta}. 
%\begin{equation}%\label{Kato-eta}
%\begin{split}
%|\eta(x,v_1)-\eta(x,v_2)|_{,t}+\div \left(\sgn(v_1-v_2)(\bG(v_1)-\bG(v_2))\right)\\
%\le \sgn(v_1-v_2)(f_1-f_2)%+|f_1-f_2|\chi_{\{u_1=u_2\}}
%\end{split}
%\end{equation}

\end{lemma}
\begin{proof}%[Proof of Lemma~\ref{contraction}]
%Assume that $E\in {\cal C}^1(\R)$. 
%$$(\nu,\sigma)\mapsto \langle Q(\lambda, \mu),\nu\otimes \sigma\rangle \in\R$$
%is a  bounded bilinear form from $\M(\R)\times \M(\R)$ to $\R$ and 
%$$x\mapsto \nu_{(t,x)}^\varepsilon\in {\cal C}^\infty_c (K,(\M(\R), \|\cdot\|_{\M}))$$
%$$x\mapsto \mu_{(t,x)}^\varepsilon\in {\cal C}^\infty_c (K,(\M(\R), \|\cdot\|_{\M}))$$
%for any compact $K\subset \R_+\times \R^N$  and then 
%$$\div\langle Q(\lambda, \mu), \nu_{(t,x)}^\varepsilon\otimes\sigma_{(t,x)}^\varepsilon\rangle
%=\langle Q(\lambda, \mu), \nabla \nu_{(t,x)}^\varepsilon\otimes\sigma_{(t,x)}^\varepsilon\rangle
%+\langle Q(\lambda, \mu), \nu_{(t,x)}^\varepsilon\otimes\nabla \sigma_{(t,x)}^\varepsilon\rangle.
%$$
%We proceed formally, however the whole proof can be done rigorously by approximating w.r.t. $x$ and $t$ 
%the terms $\langle |\eta(x,\lambda)-\eta(x,\mu)|, \nu_{(t,x)}(\lambda)\rangle$ and 
%$\langle \sgn(\lambda-\mu)(A(x,\lambda)-A(x,\mu)), \nu_{(t,x)}(\lambda)\rangle$ (analogously for $\sigma$).
%Note that, since $A$ depends on $x$, hence we cannot, following~\cite{DiPerna}, search for a measure 
%$\nu^\varepsilon\in {\cal C}^\infty(0,T;\prob\R)$ such that 
%for any  $f\in \mathcal{C}(\R)$ the following  would hold $(\omega_1^\varepsilon*\langle  f, \nu \rangle)=\langle f, \nu^{\varepsilon}
%\rangle$ for almost all $t\in \R$ and the derivatives could be  interchanged  as follows
% $\langle  f, \partial_{x_i} \nu^\varepsilon\rangle=\partial_{x_i}\langle  f, \nu^\varepsilon \rangle$ for all $x\in \R^N$. Nevertheless, this is not necessary.
If $v_1,v_2$ are entropy weak solutions, then the Dirac masses  $\delta_{v_1(t,x)}$ and $\delta_{v_2(t,x)}$ 
are corresponding entropy measure-valued solutions. 
Repeating step by step the argumentation from previous lemma we arrive at 

\begin{equation}
\begin{split}
-\int_{\Rdp} &\left(\left\langle \omega_2^{\varepsilon}*\langle |\eta(x,\lambda)-\eta(x,\mu)|, \delta_{v_1(t,x)}^\delta(\lambda)\rangle_{t}, \delta_{v_2(t,x)}^{\delta,\varepsilon}(\mu)\right\rangle\right.\\&+
\left.\left\langle \omega_2^{\varepsilon}*\langle |\eta(x,\lambda)-\eta(x,\mu)|, \delta_{v_2(t,x)}^\delta(\mu)\rangle_{t}, \delta_{v_1(t,x)}^{\delta,\varepsilon}(\lambda)\right\rangle\right)\psi\; dx\; dt \\
& +\int_{\Rdp}\left\langle  \sgn(\lambda-\mu)(A(\lambda)-A(\mu)), \delta_{v_1(t,x)}^{\delta,\varepsilon}(\lambda)\otimes \delta_{v_2(t,x)}^{\delta,\varepsilon}(\mu)\right\rangle
\cdot \nabla \psi \; dx \; dt \\
%&+
%\int_{\Rdp}\left\langle  {\cal R}_\varepsilon(t,x,\mu), \nu_{(t,x)}^{\delta,\varepsilon}(\lambda)\otimes \sigma_{(t,x)}^{\delta,\varepsilon}(\mu)\right\rangle
%\cdot \nabla \psi \; dx \; dt\\
&\ge -
\int_{\Rdp}
%  \langle \sgn(\lambda- \mu)
%(f_1(t,x,\lambda)-f_2(t,x, \mu), \delta_{v_1(t,x)}(\lambda)\otimes   \tilde\delta_{v_2(t,x)}^{\delta,\varepsilon}(\mu)\rangle
%\right.\\
%&\left.+
 \sgn(\lambda- \mu)
(f_1^n(t,x,\lambda)-f_2^n(t,x, \mu)), \delta_{v_2(t,x)}^{\delta,\varepsilon}(\mu)\otimes  \delta_{v_1(t,x)}^{\delta,\varepsilon}(\lambda)\rangle
%
%
%
%2 (\omega_1^{\delta}\cdot \omega_2^{\varepsilon})*\langle E'(\lambda- \mu)
%(f_1(t,x,\lambda)-f_2(t,x, \mu), \nu_{(t,x)}(\lambda)\otimes  \sigma_{(t,x)}(\mu)\rangle
\psi \; dx \; dt\\
&+\int_{\Rdp}
({\cal K }_{\varepsilon,\delta,n}^\lambda+{\cal K}_{\varepsilon,\delta,n}^\mu)\psi \; dx \; dt
\end{split}\label{regf00-weak}
\end{equation}
where
\begin{equation}\begin{split}
{\cal K}_{\varepsilon,\delta,n}^\lambda:=&
\langle(\omega_1^{\delta}\cdot \omega_2^{\varepsilon})*\langle E'(\lambda- \mu)
(f_1(t,x,\lambda)-f_1^n(t,x,\lambda), \delta_{v_1(t,x)}(\lambda)\rangle,   \delta^{\delta,\varepsilon}_{v_2(t,x)}(\mu)\rangle\\
+&
\langle(\omega_1^{\delta}\cdot \omega_2^{\varepsilon})*\langle E'(\lambda- \mu)
f_1^n(t,x,\lambda), \delta_{v_1(t,x)}(\lambda)\rangle,   \delta^{\delta,\varepsilon}_{v_2(t,x)}(\mu)\rangle\\
-&\langle\langle E'(\lambda- \mu)
f_1^n(t,x,\lambda), \delta^{\delta,\varepsilon}_{v_1(t,x)}(\lambda)\rangle,   \delta^{\delta,\varepsilon}_{v_2(t,x)}(\mu)\rangle
\end{split}\end{equation}
and 
\begin{equation}\begin{split}
{\cal K}_{\varepsilon,\delta,n}^\mu:=&
\langle(\omega_1^{\delta}\cdot \omega_2^{\varepsilon})*\langle E'(\mu-\lambda)
(f_2(t,x,\mu)-f_2^n(t,x,\mu), \delta_{v_2(t,x)}(\mu)\rangle, \delta^{\delta,\varepsilon}_{v_1(t,x)}(\lambda)  \rangle\\
+&
\langle(\omega_1^{\delta}\cdot \omega_2^{\varepsilon})*\langle E'(\mu-\lambda)
f^n_2(t,x,\mu), \delta_{v_2(t,x)}(\mu)\rangle,   \delta^{\delta,\varepsilon}_{v_1(t,x)}(\lambda)\rangle\\
-&\langle\langle E'(\mu-\lambda)
f^n_2(t,x,\mu), \delta^{\delta,\varepsilon}_{v_2(t,x)}(\mu)\rangle,   \delta^{\delta,\varepsilon}_{v_1(t,x)}(\lambda)\rangle
\end{split}\end{equation}

In what follows we shall only concentrate on the first integral on the right hand side of \eqref{regf00-weak}. 
The error estimates follow the same lines as \eqref{error-es}.

Since the function $ \sgn(\lambda- \mu)
(f_1^n(t,x,\lambda)-f_2^n(t,x, \mu))$ may fail to be continuous for $\lambda=\mu$, hence we shall discuss separately the integrals 
\begin{equation}
I_1^{\delta,\varepsilon}:=\int_{\{(t,x):v_1\neq v_2\}}  \langle\sgn(\lambda- \mu)
(f_1^n(t,x,\lambda)-f_2^n(t,x, \mu)), \delta_{v_2(t,x)}^{\delta,\varepsilon}(\mu)\otimes  \delta_{v_1(t,x)}^{\delta,\varepsilon}(\lambda)\rangle
\psi \; dx \; dt\\
\end{equation}
and
\begin{equation}
I_2^{\delta,\varepsilon}:=\int_{\{(t,x):v_1=v_2\}}\langle \sgn(\lambda- \mu)
(f_1^n(t,x,\lambda)-f_2^n(t,x, \mu)), \delta_{v_2(t,x)}^{\delta,\varepsilon}(\mu)\otimes  \delta_{v_1(t,x)}^{\delta,\varepsilon}(\lambda)\rangle
\psi \; dx \; dt.\\
\end{equation}
We let $\varepsilon\to0_+$ and $\delta\to0_+$. Then
\begin{equation}\begin{split}
\lim\limits_{\delta\to0}\lim\limits_{\varepsilon\to0} I_1^{\delta,\varepsilon}&=
\int_{\{(t,x):v_1\neq v_2\}}  \langle\sgn(\lambda- \mu)
(f_1^n(t,x,\lambda)-f_2^n(t,x, \mu)), \delta_{v_2(t,x)}(\mu)\otimes  \delta_{v_1(t,x)}(\lambda)\rangle
\psi \; dx \; dt\\
&=\int_{\Rdp}  \sgn(v_1- v_2)
(f_1^n(t,x,v_1)-f_2^n(t,x, v_2))
\psi \; dx \; dt.
\end{split}\end{equation}
where the last equality holds since $\sgn 0=0$.
The second integral can be estimated as follows
\begin{equation}\begin{split}
|I_2^{\delta,\varepsilon}|&\le\int_{\{(t,x):v_1=v_2\}} |\langle \sgn(\lambda- \mu)
(f_1^n(t,x,\lambda)-f_2^n(t,x, \mu)), \delta_{v_2(t,x)}^{\delta,\varepsilon}(\mu)\otimes  \delta_{v_1(t,x)}^{\delta,\varepsilon}(\lambda)\rangle|
\psi \; dx \; dt\\
&\le 
\int_{\{(t,x):v_1=v_2\}} \langle |
f_1^n(t,x,\lambda)-f_2^n(t,x, \mu)|, \delta_{v_2(t,x)}^{\delta,\varepsilon}(\mu)\otimes  \delta_{v_1(t,x)}^{\delta,\varepsilon}(\lambda)\rangle
\psi \; dx \; dt\\
\end{split}\end{equation}
and therefore 
\begin{equation}\begin{split}
\lim\limits_{\delta\to0}\lim\limits_{\varepsilon\to0} I_2^{\delta,\varepsilon}\le
 \int_{\{(t,x):v_1=v_2\}} 
 |f_1^n(t,x,v_1)-f_2^n(t,x, v_2)| \psi \; dx \; dt.
\end{split}\end{equation}
which completes the proof of point 1.

 To prove the second part of the theorem, again we shall argue on the level of measure-valued solutions. Now we will use the entropy inequalities both for convex and concave entropies. First observe that for all $\mu\in\R$
\begin{equation}
\begin{split}
 \omega_2^{\varepsilon}*&\langle (\eta(x,\lambda)-\eta(x,\mu))^+, \delta_{v_1(t,x)}^\delta(\lambda)\rangle_{t}
%, \sigma_{(t,x)}^{\delta,\varepsilon}(\mu)\rangle\\
%&+\langle \omega_2^{\varepsilon}*\langle (\eta(x,\lambda)-\eta(x,\mu))^+, \sigma_{(t,x)}^\delta(\mu)\rangle_{t}, \nu_{(t,x)}^{\delta,\varepsilon}(\lambda)\rangle\\
+\div \langle  \chi_{\{\lambda>\mu\}}(A(\lambda)-A(\mu)), \delta_{v_1(t,x)}^{\delta,\varepsilon}(\lambda)\rangle
%\otimes
% \sigma_{(t,x)}^{\delta,\varepsilon}(\mu)\rangle %+\div \langle  {\cal R}_\varepsilon, \nu_{(t,x)}^{\delta,\varepsilon}(\lambda)\otimes
 %\sigma_{(t,x)}^{\delta,\varepsilon}(\mu)\rangle
 \\
&\le  (\omega_1^{\delta}\cdot \omega_2^{\varepsilon})*\langle \chi_{\{\lambda> \mu\}}
f_1(t,x,\lambda), \delta_{v_1(t,x)}(\lambda)\rangle\rangle
\\
%&+
%\langle (\omega_1^{\delta}\cdot \omega_2^{\varepsilon})*\langle E'(\mu-\lambda)
%f(t,x,\mu), \sigma_{(t,x)}(\mu)\rangle,   \nu^{\delta,\varepsilon}_{(t,x)}(\lambda)\rangle=:I
%\\
%&+(\omega_1^{\delta}\cdot \omega_2^{\varepsilon})*\langle E'(\tilde \lambda- \mu)
%(f_1(t,x, \lambda)-f_2(t,x, \mu), \sigma_{(t,x)}(\mu)\rangle
\end{split}
\label{regf39}
\end{equation}
and all $\lambda\in\R$
\begin{equation}
\begin{split}
% {-\langle}
- \omega_2^{\varepsilon}*&\langle (\eta(x,\mu)-\eta(x,\lambda))^-, \delta_{v_2(t,x)}^\delta(\mu)\rangle_{t}
%, \sigma_{(t,x)}^{\delta,\varepsilon}(\mu)\rangle\\
%&+\langle \omega_2^{\varepsilon}*\langle (\eta(x,\lambda)-\eta(x,\mu))^+, \sigma_{(t,x)}^\delta(\mu)\rangle_{t}, \nu_{(t,x)}^{\delta,\varepsilon}(\lambda)\rangle\\
+\div \langle  \chi_{\{\mu<\lambda\}}(A(\mu)-A(\lambda)), \delta_{v_2(t,x)}^{\delta,\varepsilon}(\mu)\rangle
%\otimes
% \sigma_{(t,x)}^{\delta,\varepsilon}(\mu)\rangle %+\div \langle  {\cal R}_\varepsilon, \nu_{(t,x)}^{\delta,\varepsilon}(\lambda)\otimes
 %\sigma_{(t,x)}^{\delta,\varepsilon}(\mu)\rangle
 \\
&\ge  (\omega_1^{\delta}\cdot \omega_2^{\varepsilon})*\langle \chi_{\{\mu< \lambda\}}
f_2(t,x,\mu), \delta_{v_2(t,x)}(\mu)\rangle\rangle.
%&+
%\langle (\omega_1^{\delta}\cdot \omega_2^{\varepsilon})*\langle E'(\mu-\lambda)
%f(t,x,\mu), \sigma_{(t,x)}(\mu)\rangle,   \nu^{\delta,\varepsilon}_{(t,x)}(\lambda)\rangle=:I
%\\
%&+(\omega_1^{\delta}\cdot \omega_2^{\varepsilon})*\langle E'(\tilde \lambda- \mu)
%(f_1(t,x, \lambda)-f_2(t,x, \mu), \sigma_{(t,x)}(\mu)\rangle
\end{split}
\label{regf40}
\end{equation}
Hence multiplying \eqref{regf40} by -1 and   {adding it to} \eqref{regf39} we obtain
\begin{equation}
\begin{split}
\langle \omega_2^{\varepsilon}*&\langle (\eta(x,\lambda)-\eta(x,\mu))^+, \delta_{v_1(t,x)}^\delta(\lambda)\rangle_{t}, \delta_{v_2(t,x)}^{\delta,\varepsilon}(\mu)\rangle\\
&+\langle \omega_2^{\varepsilon}*\langle (\eta(x,\lambda)-\eta(x,\mu))^+, \delta_{v_2(t,x)}^\delta(\mu)\rangle_{t}, \delta_{v_1(t,x)}^{\delta,\varepsilon}(\lambda)\rangle\\
&+\div \langle  \chi_{\{\lambda>\mu\}}(A(\lambda)-A(\mu)), \delta_{v_1(t,x)}^{\delta,\varepsilon}(\lambda)\otimes
 \delta_{v_2(t,x)}^{\delta,\varepsilon}(\mu)\rangle %+\div \langle  {\cal R}_\varepsilon, \nu_{(t,x)}^{\delta,\varepsilon}(\lambda)\otimes
 %\sigma_{(t,x)}^{\delta,\varepsilon}(\mu)\rangle
 \\
&\le \langle (\omega_1^{\delta}\cdot \omega_2^{\varepsilon})*\langle \chi_{\{\lambda> \mu\}}
f_1(t,x,\lambda), \delta_{v_1(t,x)}(\lambda)\rangle,   \delta^{\delta,\varepsilon}_{v_2(t,x)}(\mu)\rangle
\\
&-
\langle (\omega_1^{\delta}\cdot \omega_2^{\varepsilon})*\langle \chi_{\{\lambda> \mu\}}
f_2(t,x,\mu), \delta_{v_2(t,x)}(\mu)\rangle,   \delta^{\delta,\varepsilon}_{v_1(t,x)}(\lambda)\rangle
%\\
%&+(\omega_1^{\delta}\cdot \omega_2^{\varepsilon})*\langle E'(\tilde \lambda- \mu)
%(f_1(t,x, \lambda)-f_2(t,x, \mu), \sigma_{(t,x)}(\mu)\rangle
\end{split}
%\label{regf}
\end{equation}
%
%\begin{equation}
%\begin{split}
%-\int_{\Rdp} &\left(\left\langle \omega_2^{\varepsilon}*\langle (\eta(x,\lambda)-\eta(x,\mu))^+, \delta_{v_1(t,x)}^\delta(\lambda)\rangle_{t}, \delta_{v_2(t,x)}^{\delta,\varepsilon}(\mu)\right\rangle\right.\\&+
%\left.\left\langle \omega_2^{\varepsilon}*\langle (\eta(x,\lambda)-\eta(x,\mu))^+, \delta_{v_2(t,x)}^\delta(\mu)\rangle_{,t}, \delta_{v_1(t,x)}^{\delta,\varepsilon}(\lambda)\right\rangle\right)\psi\; dx\; dt \\
%& +\int_{\Rdp}\left\langle  \chi_{\{\lambda>\mu\}}(A(\lambda)-A(\mu)), \delta_{v_1(t,x)}^{\delta,\varepsilon}(\lambda)\otimes \delta_{v_2(t,x)}^{\delta,\varepsilon}(\mu)\right\rangle
%\cdot \nabla \psi \; dx \; dt \\
%%&+
%%\int_{\Rdp}\left\langle  {\cal R}_\varepsilon(t,x,\mu), \nu_{(t,x)}^{\delta,\varepsilon}(\lambda)\otimes \sigma_{(t,x)}^{\delta,\varepsilon}(\mu)\right\rangle
%%\cdot \nabla \psi \; dx \; dt\\
%&\ge -
%\int_{\Rdp}
%  \langle  \chi_{\{\lambda>\mu\}}
%(f_1^n(t,x,\lambda)-f_2^n(t,x, \mu)), \delta^{\delta,\varepsilon}_{v_1(t,x)}(\lambda)\otimes \delta_{v_2(t,x)}^{\delta,\varepsilon}(\mu)\rangle
%%\right.\\
%%&\left.+
%% \chi_{\{\lambda>\mu\}}
%%(f_1(t,x,\lambda)-f_2(t,x, \mu), \delta_{v_2(t,x)}(\mu)\otimes \tilde \delta_{v_1(t,x)}^{\delta,\varepsilon}(\lambda)\rangle
%%\right)
%%
%%
%%
%%2 (\omega_1^{\delta}\cdot \omega_2^{\varepsilon})*\langle E'(\lambda- \mu)
%%(f_1(t,x,\lambda)-f_2(t,x, \mu), \nu_{(t,x)}(\lambda)\otimes  \sigma_{(t,x)}(\mu)\rangle
%\psi \; dx \; dt\\
%&+\int_{\Rdp}
%({\cal K }_{\varepsilon,\delta,n}^\lambda+{\cal K}_{\varepsilon,\delta,n}^\mu)\psi \; dx \; dt
%\end{split}\label{regf000-weak}
%\end{equation}
We repeat the same arguments as in the previous part of the proof. 

\end{proof}

\section{Existence of entropy weak solutions}\label{existence}
{\bf Proof of Theorem~\ref{main}.}
The proof starts with the existence of entropy measure-valued solution, then we shall show that it is  unique and is in fact an entropy weak solution. 

We construct the approximate problem. Let now $A^j$ be a sequence of smooth functions such that for every compact set $K\subset \R$ 
\begin{equation}
A^j\to A\quad {\rm strongly\ in\  }  {{\cal C}(K;\mathbb{R}^N)}.
\end{equation}
%Then we define 
%\begin{equation}\eta^{(j)}(x,z):=\eta(x,z)+\frac zj\end{equation}
%and by $\theta^{(j)}$ we mean the inverse to $\eta^{(j)}$ with respect to the second variable. 
%Note  that $\theta^{(j)}$ is Lipschitz continuous w.r.t. the second variable with Lipschitz constant less or equal to $j$. 
Let $\theta^*$ be a  minimal selection of the graph of $ \theta $.
We approximate $\theta$ in two steps. First we shall construct the Yosida approximation of $\theta$ with a parameter $\sqrt j$ and then mollify this $\theta_{\sqrt j}$ with respect to $x$ and $u$. Therefore let us define
\begin{equation}
 {J_{\frac{1}{\sqrt j}}=({\rm id}+{\frac{1}{\sqrt j}}\,\theta)^{-1}}
\end{equation}
and
\begin{equation}
 {\theta_{\frac{1}{\sqrt j}}=\sqrt{j}\ ({\rm id}-J_{\frac{1}{\sqrt j}})}.
\end{equation}
%We use a standard mollification procedure and introduce
%{\cred zrobic tak - najpierw pierwsza aproksymacje - albo aproksymacje Yosidy, albo taka jak my mielismy Lipschitzowska, zeby zapewnic, ze funkcja w zerze nie bedzie wielowartosciowa, a potem zrobic splot}
Then
\begin{equation}
\theta^{(j)}(x,u):= \int_{\R^N\times\R} \omega^{\frac{1}{j}}(x-y,u-z)\theta_{\frac{1}{\sqrt j}}(y,z)\;dy\;dz,
\end{equation}
where $\omega^{\frac{1}{j}}$ is the standard mollification kernel of radius $\frac{1}{j}$.
To provide that the approximation vanishes at zero define
\begin{equation}
\theta^j(x,u):=\theta^{(j)}(x,u)-\theta^{(j)}(x,0). 
\end{equation}
Observe that with such a choice of parameters we get 
\begin{equation}\label{theta-zero}
\theta^{(j)}(\cdot,0)\to 0\quad \mbox {a.e. in} \ \R^N
\end{equation}
 and we denote by $\eta^{j}(x,z)$ the inverse function to $\theta^{j}(x,u)$, i.e., $\eta^{j}(x,\theta^{j}(x,u))=u$.
Moreover, let
%\begin{equation}
%f^{(j)}(t,x,u):=f(t,x,u)+\frac uj
%\end{equation}
%and then define
\begin{equation}
f^{(j)}(t,x,u):= \int_{\R\times\R^N\times \R} \omega^{\frac{1}{j}}(t-s,x-y,u-z)f^{(j)}(s,y,z)\;ds\;dy\;dz
\end{equation}
and define 
\begin{equation}
f^j(t,x,u):=f^{(j)}(t,x,u)-f^{(j)}(t,x,0).
\end{equation}
Moreover, we will add a strictly dissipative perturbation term defined as follows
$$\varphi_{\ell,m}(r):=\frac{1}{\ell}\arctan(r^-)-\frac{1}{m}\arctan(r^+).$$
Hence the approximate problem has a form 
\begin{align}
 u^j_t+ \div A^j(\theta^j(x,u^j))=
f^j(t,x,u^j)+\varphi_{\ell,m}(\theta^j(x,u^j))&\quad \mbox{ on } \R_+\times \R^N,\\
%\eta(x,g^\varepsilon)=0   &\quad\mbox{ on } \Sigma_\delta=(0,T)\times \partial\Omega_\delta\\
u^j(0,\cdot)= u_0 &\quad\mbox{ on }  \R^N.
\end{align}
We will divide the proof into three steps. In the first step we shall concentrate on existence of measure-valued solutions (namely we will pass with $j\to\infty$), in the second step we will show that the measure-valued solution is indeed an entropy weak solution to the problem with a strictly dissipative  perturbation  and in the final third step we will pass to the limit with $\ell,m\to\infty$ and conclude existence of entropy weak solution to the original problem.

{\bf Step 1.} Existence of solutions is provided by the classical theory of Kru\v zkov, cf.~\cite{Kr70}. 
Since condition \eqref{betaas} holds, %then using the monotonicity of $\theta_{\sqrt j}$ we conclude
%that {$l\mapsto\theta_{\sqrt j}(\cdot,l)\in L^\infty(\R^N\times[-M,M])$ for all $M>0$. 
with the standard estimates one gets that for any $j$
 {$\theta^j$, $\frac{\partial}{\partial x_i}\theta^j=\theta_{\frac{1}{\sqrt j}}*\frac{\partial}{\partial x_i}\omega^{\frac{1}{j}}$ is bounded in $\mathbb{R}^N\times[-M,M]$ for all $M>0$ and the assumptions of
~\cite{Kr70} are satisfied.}
By Lemma~\ref{equiv}
we can define 
\begin{equation}
v^j(t,x):=\theta^j(x,u^j(t,x))
\end{equation}
which satisfies  for all nonnegative $\psi\in\D(\R\times\R^N)$
%\begin{align}\label{CP-g}
%\frac{\partial }{\partial t} \eta^\varepsilon(x, g^\varepsilon)+ \div A(x,g^\varepsilon)+ f(x,g^\varepsilon)=0&\quad \mbox{ on } \R_+\times \R^N\\
%%\eta(x,g^\varepsilon)=0   &\quad\mbox{ on } \Sigma_\delta=(0,T)\times \partial\Omega_\delta\\
%\eta(x,g(0,\cdot))= u_0 &\quad\mbox{ on }  \R^N
%\end{align}
%where
%$f(t,x,r)=f(t,x,r)+\varphi_{k,\ell}(r)$ and
%%$$\beta_{m,n}(r)=\chi_{\{Q\setminus Q\}}m(r-a(t,x))^+-n(a(t,x)-r)^+$$
%The function $g^\varepsilon$ satisfies
 the entropy inequality
 
 \begin{equation}
\begin{split}\label{entropy-v}
&\int_{\R_+\times\R^N}|\eta^j(x,v^j(t,x))-\eta^j(x,k)|\psi_{t}(t,x)\; dx \; dt\\
&\qquad+\int_{\R_+\times\R^N}\sgn(v^j(t,x)-k)(\bG^j(v^j(t,x))-\bG^j(k))\cdot \nabla \psi(t,x) \; dx \; dt\\
&\qquad +\int_{\R_+\times\R^N}\sgn (v^j(t,x)-k) (f^j(t,x,\eta^j(x,v^j))+\varphi_{\ell,m}(v^j))\psi \;dx\;dt\\
&\qquad+\int_{\R^N}|u_0(x)-\eta^j(x,k)|\psi(0,x) \;dx\ge 0.
\end{split}
\end{equation}

%\begin{equation}\label{entropy-epsilon}
%\begin{split}
%\int_Q\{|\eta^\varepsilon(x,g^\varepsilon)-\eta^\varepsilon(x,k)|\xi_t&+{\rm sgn}(g^\varepsilon-k)
%(A(x,g^\varepsilon)-A(x,k))\nabla \xi \\
%&+{\rm sgn}(g^\varepsilon-k)(\div A(x,k)-f(x,g^\varepsilon))\xi\}\ge0
%\end{split}
%\end{equation}
Since $u^j$ is bounded in $L^\infty(\R^+\times \R^N)$, then  by \eqref{h1h2} the sequence $v^j$
is also bounded. 
From the entropy inequality \eqref{entropy-v} we want to pass to the following entropy inequalities

 \begin{equation}\label{ap-plus}
\begin{split}
&\int_{\R_+\times\R^N}(\eta^j(x,v^j(t,x))-\eta^j(x,k))^+\psi_{t}(t,x)\; dx \; dt\\
&\qquad+\int_{\R_+\times\R^N}\chi_{\{v^j(t,x)>k\}}(\bG^j(v^j(t,x))-\bG^j(k))\cdot \nabla \psi(t,x) \; dx \; dt\\
&\qquad +\int_{\R_+\times\R^N}\chi_{\{v^j(t,x)>k\}} (f^j(t,x,\eta^j(x,v^j))+\varphi_{\ell,m}(v^j))\psi \;dx\;dt\\
&\qquad+\int_{\R^N}{(u_0(x)-\eta^j(x,k))}^+\psi(0,x) \;dx\ge 0
\end{split}
\end{equation}

and 

 \begin{equation}\label{ap-minus}
\begin{split}
&\int_{\R_+\times\R^N}(\eta^j(x,v^j(t,x))-\eta^j(x,k))^-\psi_{t}(t,x)\; dx \; dt\\
&\qquad+\int_{\R_+\times\R^N}\chi_{\{v^j(t,x)<k\}}(\bG^j(v^j(t,x))-\bG^j(k))\cdot \nabla \psi(t,x) \; dx \; dt\\
&\qquad +\int_{\R_+\times\R^N}\chi_{\{v^j(t,x)<k\}} (f^j(t,x,\eta^j(x,v^j))+\varphi_{\ell,m}(v^j))\psi \;dx\;dt\\
&\qquad+\int_{\R^N}{(u_0(x)-\eta^j(x,k))}^-\psi(0,x)\; dx\le 0
\end{split}
\end{equation}
satisfied for all nonnegative $\psi\in\D(\R\times\R^N)$. For this purpose we first choose in \eqref{entropy-v} $k=\|v^j\|_{L^\infty}$ and $k=-\|v^j\|_{L^\infty}$, which allows to conclude  that the problem
\begin{equation}\label{dist}
\begin{split}
\eta^j(x,v^j)_t+\div A^j(v^j)&=f^j(t,x,\eta(x,v^j))+\varphi_{\ell,m}(v^j),\\
v^j(0,x)&=\theta^j(x,u_0)
\end{split}
\end{equation}
is satisfied in  a distributional sense. Obviously the following problem 
\begin{equation}\label{k}
\begin{split}
\eta^j(x,k)_t+\div A^j(k)&=0,
%v^j(0,x)&=k.
\end{split}
\end{equation}
with initial condition $k$ is satisfied in $\D'(\R\times\R^N)$.
Hence a linear combination of \eqref{entropy-v}, \eqref{dist} and \eqref{k} allows to conclude \eqref{ap-plus}
and \eqref{ap-minus}.

We want to pass to the limit with $j\to\infty$ in \eqref{ap-plus} (and  \eqref{ap-minus} respectively, which we do not present in detail since it is  easily concluded from the first part). 
Obviously, there exists a subsequence (labelled the same) and $v\in L^\infty(\R^+ \times \R^N)$ such that 
\begin{equation}\label{conv-g}
v^j\weakstar v\quad \mbox{in } L^\infty(\R^+ \times \R^N).
\end{equation}
%Then, we conclude the existence of a Young measure $\nu_{t,x}$ associated to the sequence 
%$\{v^j\}$ and pass to the limit in ~\eqref{entropy-v} to get
%\begin{equation}\begin{split}
%\int_Q\{\langle |\eta(x,\lambda)-\eta(x,k)|, \nu_{(t,x)}(\lambda)\rangle \xi_t+\left\langle{\rm sgn}(\lambda-k)
%(A(x,\lambda)-A(x,k)), \nu_{(t,x)}(\lambda)\right\rangle \nabla \xi\\
%%+ [\langle{\rm sgn}(\lambda-k)\div A(x,k),\nu_{\{t,x\}}(\lambda)\rangle
%-\langle {\rm sgn}(\lambda-k) f(x,\lambda),\nu_{(t,x)}(\lambda)\rangle] \xi\}\ge0
%\end{split}\end{equation}
Moreover there exists a Young measure $\nu_{(t,x)}$ associated to the subsequence $v^j$. 
In the remaining part of this step of the proof we will show that $\nu$ is an entropy measure-valued solution
in the sense of Definition~\ref{DF2}.  
To pass to the limit in  \eqref{ap-plus}   (\eqref{ap-minus} follows analogously) with the first terms on the right-hand side we first make an observation on $\theta^j$, namely due to \eqref{theta-zero} for all $(x,u)\in\R^N\times\R$ where $\theta$ is single-valued and continuous with respect to $u$
\begin{equation}
 {\theta^j\to\theta^{\ast}} %\quad\mbox{for a.a.  }(x,u)\in\R^N\times\R
\end{equation}
 {a.e. in $\mathbb{R}^N$. Hence there exists $M\subset\mathbb{R}^N$ such that $|M|=0$ such that 
\[\theta^j(x,\cdot)\rightarrow\theta^{\ast}(x,\cdot)\]
for all $x\in\mathbb{R}^N\setminus M$.}
 %Moreover, the complement of this set is of measure zero. }
The strict monotonicity of $\theta^j$ with respect to the last variable allows to conclude with help of Proposition~\ref{inverse-conv} for a.a. $x\in\R^N$ the locally uniform convergence of $\eta^j(x,\cdot)$. Define 
\begin{equation}
\zeta^j(x,s):=\left(\eta^j(x,s)-\eta^j(x,k)\right)^+
\end{equation}
and 
\begin{equation}\label{conv-eta}
\begin{split}
\lim\limits_{j\to\infty}\int_{\R_+ {\times \R^N\setminus M}}& \zeta^j(x,v^j)\;dx\;dt
=
\lim\limits_{j\to\infty}
\langle \zeta^j, v^j\rangle\\&=
\int_{\R_+\times \R^N}\int_\R(\eta(x,\lambda)-\eta(x,k))^+d\nu_{(t,x)}(\lambda) \;dx\;dt
%\langle \eta,U(g)\otimes \psi\rangle,
\end{split}\end{equation}
where the duality pairing is understood between the spaces
$L^1(\R^d;\mathcal{C}((-R,R);\R))$ and $L^\infty_{w}(\R^d;{\mathcal M}([-R,R]))$. 
The limit passage in the second term of  \eqref{ap-plus} and  \eqref{ap-minus} follows the same lines as in~\cite{BuGwSw2013}.

 We direct our attention to the limit passage in the term containing $f^j$. The main problem is the appearance of a discontinuous function 
$$\lambda\mapsto\chi_{\{\lambda>\mu\}}(f^j(t,x,\eta^j(x,\lambda))+\varphi_{\ell,m}(\lambda)).$$ For this purpose  
we shall construct a family of functions which  allow to estimate the discontinuous term. We will call it  
$\chi^\gamma_{\{\lambda>\mu\}}$ and define as follows:
for $\mu\ge0$
\begin{equation}
\chi^{\gamma,+}_{\{\lambda>\mu\}}(\lambda):=\left\{
\begin{array}{rcl}
\chi_{\{\lambda>\mu\}}&{\rm for}& \lambda<\mu, \, \lambda\ge \mu+\gamma,\\[1ex]
{\rm  {affine}}&{\rm for} &\mu\le\lambda<\mu+\gamma.
\end{array}\right.\end{equation}
For $\mu<0$
\begin{equation}
\chi^{\gamma,-}_{\{\lambda>\mu\}}(\lambda):=\left\{
\begin{array}{rcl}
\chi_{\{\lambda>\mu\}}&{\rm for}& \lambda<\mu-\gamma, \, \lambda\ge \mu,\\[1ex]
{\rm  {affine}} &{\rm for} &\mu-\gamma\le\lambda<\mu.
\end{array}\right.\end{equation}
Note that since $f+\varphi_{\ell,m}$ are dissipative, then the above definition of 
$\chi^\gamma_{\{\lambda>\mu\}}$ provides that
\begin{equation}
\chi^\gamma_{\{\lambda>\mu\}} (f(t,x,\lambda)+\varphi_{\ell,m}(\lambda))\ge 
\chi_{\{\lambda>\mu\}} (f(t,x,\lambda)+\varphi_{\ell,m}(\lambda))
\end{equation}
for any $\lambda\in\R$, therefore inequality \eqref{ap-plus} with $\chi^\gamma_{\{\lambda>\mu\}}$ instead of $\chi_{\{\lambda>\mu\}}$ in the third term on the left-hand side holds. 
The convergence of convolutions and the dissipativity/monotonicity of $f$,  $f^j$ and $\eta, \eta^j$ provide that $f^j(t,x,\eta^j(x,\lambda))$ converges  {a.e} with respect to $t$ and $x$ and uniformly with respect to $\lambda$ on a bounded interval $[-R,R]$  to the function $f$, see Proposition~\ref{monotone-conv}, namely
\begin{equation}
f^j(\cdot,\cdot, \eta^j)\to f(\cdot,\cdot, \eta)\quad \mbox{strongly in }\quad L^1_{loc}(\R_+\times\R^N;
{\cal C}([-R,R]).
\end{equation}
We obtain that 
\begin{equation}\begin{split}
\lim\limits_{j\to\infty}\int_{\R_+\times\R^N}&\chi^\gamma_{\{v^j(t,x)>k\}} (f^j(t,x,\eta^j(x,v^j))+\varphi_{\ell,m}(v^j))\psi \;dx\;dt\\&= \int_{\Rdp}
\langle \chi^\gamma_{\{\lambda>\mu\}} (f(t,x,\lambda)+\varphi_{\ell,m}(\lambda)),\nu_{(t,x)}(\lambda)\rangle\psi \; dx\; dt.
\end{split}\end{equation}
Then we pass with $\gamma\to0_+$. The limit passage is obvious for those $ {\mu}$ 
that $\nu_{(t,x)}(\{ \mu\})\stackrel{{\rm a.e.}}{=}0$ on  $\supp\psi$.
Let again $\Omega_\psi:=\supp \psi$, where $\psi$ has compact support in $\R_+\times\R^N$ hence $|\Omega_\psi|<\infty$.  
We shall now concentrate on showing that  the set 
$$I:=\{\mu\in\R: |\{(t,x)\in\Omega_\psi: \nu_{(t,x)}(\{\mu\})>0\}|>0\}$$ 
is at most countable. Indeed, assume the opposite. 
If $\nu_{(t,x)}(\{\mu\})>0$ on some subset of $\R_+\times\R^N$  of positive measure, then 
$\int_{\Omega_\psi}\nu_{(t,x)}(\{\mu\})\;dx\;dt>0$, but also 
$$\sum\limits_{\mu\in I}\int_{\Omega_\psi}\nu_{(t,x)}(\{\mu\})\;dx\;dt\le \int_{\Omega_\psi}\int_\R 1
\;d\nu_{(t,x)}\;dx\;dt.$$
Since the set  $I$ is not countable, then the series diverges, but we know that the Young measure $\nu$ is a probability measure, therefore the right-hand side equals to $|\Omega_\psi|$ and we obtain a contradiction. 
Consequently the set $\R\setminus I$ is a dense set in $\R$.
 We conclude that for all $\mu\in\R\setminus I$ and all nonnegative $\psi\in\D(\R_+\times\R^N)$
\begin{equation}
\begin{split}
&\int_{\Rdp}\langle (\eta(x,\lambda)-\eta(x,\mu))^+,\nu_{(t,x)}(\lambda)
\rangle \psi_{t}(t,x)\; dx \; dt \\
&\quad + \int_{\Rdp}\langle \chi_{\{\lambda>\mu\}}(A(\lambda)-A(\mu))
 ,\nu_{(t,x)}(\lambda)\rangle \cdot \nabla \psi (t,x) \; dx\; dt \\
&\quad + \int_{\Rdp}
\langle \chi_{\{\lambda>\mu\}} (f(t,x,\lambda)+\varphi_{\ell,m}(\lambda)),\nu_{(t,x)}(\lambda)\rangle\psi \; dx\; dt\ge 0 . \label{EQ+2}
\end{split}
\end{equation}
To claim that the above inequality holds for all $\mu\in\R$ observe that the function
$$\mu\mapsto\int_{\Rdp}\langle (\eta(x,\lambda)-\eta(x,\mu))^+,\nu_{(t,x)}(\lambda)
\rangle \psi_{t}(t,x)\; dx \; dt $$
as well as
$$\mu\mapsto\int_{\Rdp}\langle \chi_{\{\lambda>\mu\}}(A(\lambda)-A(\mu))
 ,\nu_{(t,x)}(\lambda)\rangle \cdot \nabla \psi (t,x) \; dx\; dt $$
are continuous w.r.t. $\mu$. 
Observe now the function 
\begin{equation}\label{efy}
\mu\mapsto \int_{\Rdp}
\langle \chi_{\{\lambda>\mu\}} (f(t,x,\lambda)+\varphi_{\ell,m}(\lambda)),\nu_{(t,x)}(\lambda)\rangle\psi \; dx\; dt, 
\end{equation}
which is not continuous w.r.t. $\mu$, but one can notice it is decreasing/increasing depending on the sign of $\mu$. For this purpose let us split the integral as follows
\begin{equation}\begin{split}
\int_{\Rdp}
\int_{\R_+} \chi_{\{\lambda>\mu\}} (f(t,x,\lambda)+\varphi_{\ell,m}(\lambda)) \;d\nu_{(t,x)}(\lambda) \psi \; dx\; dt\\+
\int_{\Rdp}
\int_{\R_-} \chi_{\{\lambda>\mu\}} (f(t,x,\lambda)+\varphi_{\ell,m}(\lambda)) \;d\nu_{(t,x)}(\lambda) \psi \; dx\; dt.
\end{split}\end{equation}
Depending on the sign of $\mu$, always  the terms $\chi_{\{\lambda>\mu\}}$ in one of  the above integrals will be constant.  In the second integral, because of the dissipativity of $f$ and $\varphi_{\ell,m}$, we know the sign of the integrand, which allows to claim that the function \eqref{efy} is monotone w.r.t. $\mu$. 
Therefore if we take $\mu\in I, \mu>0$, then one can find a sequence $\mu^n$ such that
\begin{equation}\begin{split}
\lim\limits_{\mu^n\to\mu^-}&\int_{\Rdp}
\int_{\R}\chi_{\{\lambda>\mu^n\}}f(t,x,\lambda)d\nu_{(t,x)}(\lambda)\ge 
\int_{\Rdp}
\int_{\R}\chi_{\{\lambda>\mu\}}f(t,x,\lambda)d\nu_{(t,x)}(\lambda)\\
&\ge
\lim\limits_{\mu^n\to\mu^+}\int_{\Rdp}
\int_{\R}\chi_{\{\lambda>\mu^n\}}f(t,x,\lambda)d\nu_{(t,x)}(\lambda).
\end{split}\end{equation}
For $\mu\in I, \,\mu<0$ the inequalities hold in an opposite direction. Analogously one can show that \eqref{EQ-} holds. 

{\bf Step 2.} Let now $\nu, \sigma$ be two entropy measure-valued solutions. By Lemma~\ref{contraction} we obtain that \eqref{begin3} holds with
 $f(t,x,\lambda)+\varphi_{\ell,m}$ instead of $f(t,x,\lambda)$. 
Let $0<\varepsilon<t_0<T<\infty$ be arbitrary. We define an affine $ {\psi^1_{\varepsilon,t_0}}$ as follows
\begin{equation*}
 {\psi^1_{\varepsilon,t_0}}(t):=\left\{ \begin{aligned}&0 &&t\in[0,t_0-\varepsilon)\cup [T,\infty),\\
&\frac{t-t_0+\varepsilon}{\varepsilon} &&t\in (t_0-\varepsilon,t_0),\\
&\frac{T-t}{T-t_0} &&t\in (t_0,T).
\end{aligned}\right.
\end{equation*}
Let $\psi_2^n \in \mathcal{D}(\R^d)$ be arbitrary such that $\|\psi_2^n\|_{\infty}\le 1.$ Then we set $\psi(t,x):= {\psi^1_{\varepsilon,t_0}(t)}\psi_2^n(x)$ in \eqref{begin3} (it is a possible test function since we can mollify $\psi_1$ and then pass to the limit). Hence, using for simplicity the notation $f_{\ell,m}(t,x,\lambda):=f(t,x,\lambda)+
\varphi_{\ell,m}(\lambda)$ and $Q(\lambda,\mu)=\sgn(\lambda-\mu)(A(\lambda)-A(\mu))$
\begin{equation*}
\begin{split}
&\frac{1}{T-t_0}\int_{t_0}^T\int_{\R^N}\langle |\eta(x,\lambda)-\eta(x,\mu)|, \nu_{(t,x)}(\lambda)\otimes
\sigma_{(t,x)}(\mu)\rangle \psi^n_2(x)\; dx\; dt \\
& \le\frac{1}{\varepsilon}\int_{t_0-\varepsilon}^{t_0}\int_{\R^N}\langle |\eta(x,\lambda)-\eta(x,\mu)|, \nu_{(t,x)}(\lambda)\otimes
 \sigma_{(t,x)}(\mu)\rangle \psi^n_2(x)\; dx \; dt\\
&+ \int_{t_0-\varepsilon}^{T}\int_{\R^N}\langle Q(\lambda,\mu), \nu_{(t,x)}(\lambda)
\otimes \sigma_{(t,x)}(\mu)\rangle \cdot \nabla \psi^n_2(x)\psi_1(t) \; dx \; dt\\
%&\quad =:I_1(\varepsilon,t_0,n)+I_2(\varepsilon,t_0, T,n).
&+ \int_{t_0-\varepsilon}^{T}\int_{\R^N}\langle \sgn(\lambda-\mu)\left(f_{\ell,m}(t,x,\lambda)-f_{\ell,m}(t,x,\mu))\right),
 \nu_{(t,x)}(\lambda)
\otimes \sigma_{(t,x)}(\mu)\rangle \cdot \psi^n_2(x) {\psi^1_{\varepsilon,t_0}} \; dx \; dt\\
\end{split}%\label{begin2}
\end{equation*}
Our goal is to let $\varepsilon \to 0_+$, and next $t_0\to0_+$. Because of the initial condition  {(see \eqref{ic})} and continuity of the solution in appropriate topology the first term on the right-hand side above will vanish. Considerations concerning the left-hand side and  {the second term} on the right-hand side follow the same lines as in  \cite{BuGwSw2013}.
 {There is no problem to pass to the limit in the term with $f_{\ell,m}$}. %Assume that $\supp \psi^n_2\subset K$, where $K$ is a compact subset of $\R^N$ and  let $\varepsilon \to 0_+$, and next $t_0\to0_+$.
%  Then one easily shows that
%\begin{equation}\begin{split}
%\lim_{\varepsilon \to 0_+}&
% \int_{t_0-\varepsilon}^{T}\int_{\R^N}\langle \sgn(\lambda-\mu)\left(f_{\ell,m}(t,x,\lambda)-f_{\ell,m}(t,x,\mu))\right),
% \nu_{(t,x)}(\lambda)
%\otimes \sigma_{(t,x)}(\mu)\rangle \cdot \psi^n_2(x)\psi_1(t) \; dx \; dt\\
%&= \int_{t_0}^{T}\int_{\R^N}\langle \sgn(\lambda-\mu)\left(f_{\ell,m}(t,x,\lambda)-f_{\ell,m}(t,x,\mu))\right),
% \nu_{(t,x)}(\lambda)
%\otimes \sigma_{(t,x)}(\mu)\rangle \cdot \psi^n_2(x)\psi_1(t) \; dx \; dt.
%%&\le
%% \int_{t_0}^{T}\int_{\R^N}\langle \sgn(\lambda-\mu)\left(f_{\ell,m}(t,x,\lambda)-f_{\ell,m}(t,x,\mu))\right),
%% \nu_{(t,x)}(\lambda)
%%\otimes \sigma_{(t,x)}(\mu)\rangle \cdot \psi^n_2(x)\psi_1(t) \; dx \; dt\
%%\end{split}\end{equation}
%%where
%%\begin{equation}
%%\begin{split}
%%I_2(t_0, T,n)&:=\int_{t_0}^{T}\int_{\R^N}\langle \bQ(\lambda,\mu), \nu_{(t,x)}(\lambda)
%%\otimes \sigma_{(t,x)}(\mu)\rangle \cdot \nabla \psi^n_2(x)\psi_1(t)\; dx \; dt\\
%%&\le \int_{0}^{T}\int_{\R^N}\langle |\bQ(\lambda,\mu)|, \nu_{(t,x)}(\lambda)
%%\otimes \sigma_{(t,x)}(\mu)\rangle |\nabla \psi^n_2(x)|\; dx \; dt,
%\end{split}
%\end{equation}
%and where  we used the Jensen inequality for estimating the term on the right hand side.
 {For arbitrary $\psi^n_2 \in \mathcal{D}(\R^N)$ such that $\|\psi_2^n\|_{\infty}\le 1$ and any $T>0$ at the limit we find} %\marginpar{wrong estimate}
\begin{equation}
\begin{split}
&\int_{0}^T\int_{\R^N}\langle |\eta(x,\lambda)-\eta(x,\mu)|, \nu_{(t,x)}(\lambda)\otimes
\sigma_{(t,x)}(\mu)\rangle \psi^n_2(x)\; dx\; dt \\
&\le T\int_{0}^{T}\int_{\R^N}\langle |\bQ(\lambda,\mu)|, \nu_{(t,x)}(\lambda)
\otimes \sigma_{(t,x)}(\mu)\rangle |\nabla \psi^n_2(x)| \; dx \; dt\\
&+T {\int_{0}^{T}}\int_{\R^N}\langle \sgn(\lambda-\mu)\left(f_{\ell,m}(t,x,\lambda)-f_{\ell,m}(t,x,\mu))\right),
 \nu_{(t,x)}(\lambda)
\otimes \sigma_{(t,x)}(\mu)\rangle \cdot \psi_1(t)\psi^n_2(x) \; dx \; dt.
\end{split}\label{begin31}
\end{equation}
 {Where $\psi_1(t)=1-\frac{t}{T}$ for $t\in[0,T]$.}
%Using the definition of $\bQ$ and the triangle inequality we observe that for almost all $(t,x)\in \Rdp$
%\begin{align*}
%\langle |\bQ(\lambda,\mu)|, \nu_{(t,x)}(\lambda)
%\otimes \sigma_{(t,x)}(\mu)\rangle \le \langle |\bA_1(\lambda)|, \nu_{(t,x)}(\lambda)
%\rangle
%+\langle |\bA_2(\mu)|, \sigma_{(t,x)}(\mu)\rangle.
%\end{align*}
%Hence, by \eqref{bcc} we conclude that
%\begin{equation}
%\langle |\bQ(\lambda,\mu)|, \nu_{(t,x)}(\lambda)
%\otimes \sigma_{(t,x)}(\mu)\rangle \in L^1(0,T; L^p(\Rd)). \label{LpLp}
%\end{equation}
Note that the term on the left-hand side is nonnegative. Because of the growth conditions that were assumed on $A$, we conclude that $\langle |Q(\lambda,\mu)|,\nu_{(t,x)}(\lambda)\otimes\sigma_{(t,x)}(\mu)\rangle\in L^1(0,T;L^p(\R^N))$. Finally, we define a monotone  sequence $\psi^n_2\nearrow1$ of smooth nonnegative compactly supported functions as $\psi^n_2(x):=1$ in $B(0,n)$, $\psi^n_2(x):=0$ for
$x\in \R^N\setminus B(0,2n)$ such that $|\nabla \psi^n_2|\le
\frac{c}{n}$. For handling the flux term one immediately observes that
$$
\int_{\R^N}|\nabla \psi^n_2|^{q}\; dx \le C \qquad \textrm{ for all } q\ge N,
$$
and \begin{equation}
|\nabla \psi^n|\rightharpoonup^* 0 \textrm{ weakly}^* \textrm{ in } L^{\infty}(0,T; L^q(\R^N)) \qquad \textrm{ for all } q\ge N,\label{2Lp}
\end{equation} 
which is enough that this term vanishes. 
%Hence, using \eqref{LpLp} and the weak$^*$ convergence \eqref{2Lp}, we see that the right hand side of \eqref{begin3} tends to $0$ as $n\to \infty$. 
With the monotone convergence
theorem we conclude that
\begin{equation}
\begin{split}
0\le \int_{0}^{T}\int_{\R^N}\langle \sgn(\lambda-\mu)\left(f_{\ell,m}(t,x,\lambda)-f_{\ell,m}(t,x,\mu))\right),
 \nu_{(t,x)}(\lambda)
\otimes \sigma_{(t,x)}(\mu)\rangle \psi_1(t)  \; dx \; dt\le 0.
\end{split}\label{begin4}
\end{equation}
Because of the strict dissipativity of the function $f_{\ell,m}$  the last inequality is strict except of the diagonal. Since the left-hand side is nonnegative,   
there exists a function %\marginpar{\tt think about the regularity}
\begin{equation}
v\in L^\infty(\R_+\times \R^N)
\end{equation}
such that 
\begin{equation}
\nu_{t,x}=\sigma_{t,x}=\delta_{v(t,x)}\quad\mbox{for a.a } (t,x)\in\Rdp.
\end{equation}
Hence we conclude that for each $\ell,m$ there exists an entropy weak solution. 

{\bf Step 3.} In the final step we will pass with $\ell,m\to\infty$. Let then $v_{\ell,m}$ and 
$v_{\ell,m'}$ be entropy weak  solutions to the problems with a righ-hand side 
$f+\varphi_{\ell,m}$ and $f+\varphi_{\ell,m'}$ respectively with $m'>m$. We will now use inequality \eqref{begin2} for the solutions $v_{\ell,m}$ and 
$v_{\ell,m'}$. We proceed with choosing a test
function in the same way and limit passage with $\varepsilon\to0_+$ and next $t_0\to0_+$ as in the previous step. Hence
\begin{equation}
\begin{split}
0&\le\int_{\Rdp} (\eta(x,v_{\ell,m})-\eta(x,v_{\ell,m'}))^+\; dx\; dt\\
%&\quad + \int_{\Rdp} \chi_{\{v_1>v_2\}}(A(v_1)-A(v_2)) \cdot \nabla \psi(t,x)\; dx \; dt \\
&\quad \le T \int_{\Rdp} \chi_{\{v_{\ell,m}>v_{\ell,m'}\}}(f_{\ell,m}(t,x,v_{\ell,m})-f_{\ell,m'}(t,x,v_{\ell,m'}))
\psi_1(t)\;dx\;dt\\
&\quad  + T\int_{\{(t,x):v_{\ell,m}=v_{\ell,m'}\}} (f_{\ell,m}(t,x,v_{\ell,m})-f_{\ell,m'}(t,x,v_{\ell,m'}))^+\psi_1(t)\;dx\;dt. 
\label{be}
\end{split}
\end{equation}
The second term on the right-hand side can be neglected since
\begin{equation}\begin{split}
\chi_{\{v_{\ell,m}=v_{\ell,m'}\}} &(f_{\ell,m}(t,x,v_{\ell,m})-f_{\ell,m'}(t,x,v_{\ell,m'}))^+
\\
&=\chi_{\{v_{\ell,m}=v_{\ell,m'}\}}\left(f_{\ell,m'}(t,x,v_{\ell,m})
+\left(\frac 1{m'}-\frac 1m\right) \arctan(v_{\ell,m}^+)-f_{\ell,m'}(t,x,v_{\ell,m'})\right)^+\\
&=\chi_{\{v_{\ell,m}=v_{\ell,m'}\}}\left(\left(\frac 1{m'}-\frac 1m\right) \arctan(v_{\ell,m}^+)\right)^+=0.
\end{split}\end{equation}
Observe now the first term on the right-hand side 
\begin{equation}\begin{split}
&\chi_{\{v_{\ell,m}>v_{\ell,m'}\}}(f_{\ell,m}(t,x,v_{\ell,m})-f_{\ell,m'}(t,x,v_{\ell,m'}))\\
&\quad=\chi_{\{v_{\ell,m}>v_{\ell,m'}\}}(f_{\ell,m'}(t,x,v_{\ell,m})
+\left(\frac 1{m'}-\frac 1m\right) \arctan(v_{\ell,m}^+)-f_{\ell,m'}(t,x,v_{\ell,m'}))\\
&\quad\le\chi_{\{v_{\ell,m}>v_{\ell,m'}\}}(f_{\ell,m'}(t,x,v_{\ell,m})
-f_{\ell,m'}(t,x,v_{\ell,m'}))\le0
\end{split}\end{equation}
where the last inequality holds since $m'>m$, $\arctan(v_{\ell,m}^+)\ge0$ and the function $f_{\ell,m'}$ is 
dissipative. 
Therefore, since $\psi_1(t)$ is nonnegative,  then
\begin{equation}
\chi_{\{v_{\ell,m}>v_{\ell,m'}\}}(f_{\ell,m'}(t,x,v_{\ell,m})
-f_{\ell,m'}(t,x,v_{\ell,m'}))=0 \quad{\rm a.e.\ in}\ \R_+\times\R^N.
\end{equation}
Strict dissipativity of $f_{\ell,m'}$ allows to conclude that  
\begin{equation}
v_{\ell,m}\le v_{\ell,m'}
\end{equation}
 In the same manner, choosing $\ell'>\ell$ one shows that 
 \begin{equation}\label{ell}
 v_{\ell',m}\le v_{\ell,m},
 \end{equation}
  where 
$ v_{\ell',m},  v_{\ell,m}$ are entropy weak  solutions to the problems with a righ-hand side 
$f_{\ell',m}$ and $f_{\ell,m}.$ 
We will pass to the limit with $m\to\infty$ and then with $\ell\to\infty$. The monotonicity provides that  for each $\ell$ there exists a limit   $v_\ell$
%and $v_{\ell'}$ 
such that 
\begin{equation}
v_{\ell,m}\to v_\ell\quad {\rm a.e.\ in}\ \R_+\times\R^N.
\end {equation}
Hence, if we denote $v_{\ell'}$ a limit of a sequence $v_{\ell',m}$, then from \eqref{ell} we conclude that 
\begin{equation}
v_\ell\le v_{\ell'}
\end{equation} 
for $\ell'>\ell$. Hence as $\ell\to\infty$
\begin{equation}
v_{\ell}\to v\quad {\rm a.e.\ in}\ \R_+\times\R^N.
\end {equation}
\subsection{Uniqueness of  entropy weak solutions}
 {Using the local comparison principle of Lemma \ref{22} we obtain uniqueness of the entropy weak solution. Let us assume that $u_1$ and $u_2$ are entropy weak solutions to \eqref{E1} in the sense of Definiton \ref{weak}. Then we take $\psi=\psi^1_{\varepsilon,t_0}(t)\psi_2^n(x)$ as a test function in \eqref{begin1}, where $\psi^1_{\varepsilon,t_0}$ and $\psi_2^n$ are defined as in the proof of Theorem \ref{existence} and we repeat the argumentation of this proof, Step 2 to pass to the limit with $\varepsilon\rightarrow 0^+$ and $t_0\downarrow 0$ using the initial condition. Finally we choose $\psi_2^n(x)$ to be a smooth approximation of $\chi_{\mathbb{R}^N}$ and pass to the limit with $n\rightarrow\infty$ repeating the arguments of the proof of Theorem \ref{existence}, Step 2.}  

\appendix

\section{Equivalent notions of entropy solutions}\label{EqN}
In this section we concentrate on relations between different notions of entropy weak solutions for the flux function $\Phi$ in a form $\Phi(x,u)=\bG(\theta(x,u))$ with $\bG, \theta$ satisfying  { (H1)}--{(H3)} with an additional condition that $\bF$  is sufficiently regular in both variables. This relations play an important role on the level of approximations, namely after passing from discontinuous flux to sufficiently smooth one. We formulate the lemma collecting the relations between different notions of solutions.

%In this section,  we define precisely  what we mean by the Kru\v{z}kov entropy solution to \eqref{E1}--\eqref{Kruz} or to \eqref{Kruz2} for smooth fluxes $\bF$ and show several equivalent notions to such a solution for $\bF$ satisfying in addition {\bf (A1)}--{\bf(A3)}.  There is a straightforward connection
%between the Kru\v{z}kov entropy solution and the entropy weak solution in the meaning of  present considerations (see {\bf (N4)} below) in case of one dimension. Therefore we treat this situation separately. First, we formulate the main result of this section  and then provide two separate proofs. The first one for dimension one, that is only the combination of the previous results and the second one for any dimension that is based on the kinetic formulation.
\begin{lemma}\label{equiv}
Let $ \Phi, f$ satisfy the assumptions (H1)--(H4) and assume that
$\bG\in{\cal C}^1(\R)$,  $\theta$ is continuous in $u$ and continuously differentiable in $x$. 
%smooth inverse $U^{-1}$ and let $u_0\in L^{\infty}_{loc}(\mathbb{R}^d)$. 
Assume that $u\in L^{\infty}_{loc}(\mathbb{R}\times\R^N)$ is given and define
\begin{align}
v(t,x)&:=\theta(x,u(t,x)), \label{v}\\
%g(t,x)&:=U^{-1}(v(t,x)).\label{w}
\end{align}
Then the following  statements are equivalent.\\
%{\tt Add also Audusse-Perthame and discuss the equivalence for smooth functions}\\
{\bf(N1)} For all $k\in \mathbb{R}$ and all nonnegative $\psi\in \mathcal{D}(\R\times\mathbb{R}^{N})$ there holds
    \begin{equation}\label{Kruzw}
    \begin{split}
    &\int_{\R_+\times\R^N}|u(t,x)-k|\psi_{t}(t,x) -\sgn(u(t,x)-k)\div \Phi(x,k)\psi(t,x)\; dx \; dt\\
    &\qquad +\int_{\R_+\times\R^N}\sgn(u(x,t)-k) (\Phi(x,u(x,t))-\Phi(x,k))\cdot \nabla \psi(x,t)\; dx \; dt\\
    &\qquad +\int_{\R_+\times\R^N}\sgn (u(t,x)-k) f(t,x,u)\psi \;dx\;dt+\int_{\R^N}|u_0(x)-k|\psi(0,x)\; dx\ge0.
    \end{split}
    \end{equation}
%%    {\bf(N2)} (Only valid for $d=1$).  Assume that $F$ satisfies \eqref{A1}--\eqref{A3}.  For  all $k_\alpha(x)$ satisfying $F(x,k_\alpha(x))=\alpha$ and all nonnegative $\psi\in \mathcal{D}(\mathbb{R}^{2})$ there
%    holds\footnote{As mentioned before, conditions \eqref{A1}--\eqref{A3} are less general than in
%    Ref.~\refcite{AudussePerthame} and therefore we can indeed write that the identity holds for all
%    $k_\alpha(x)$. Otherwise the choice should be restricted.}
%    \begin{equation}\label{Kruzw+}
%    \begin{split}
%    &\int_{\mathbb{R}^{d+1}_{+}}|u(t,x)-k_\alpha(x)|\psi_{,t}(t,x)\; dx \; dt\\
%    &\qquad +\int_{\mathbb{R}^{d+1}_{+}}\sgn(u(t,x)-k) (F(x,u(t,x))-F(x,k_\alpha(x)))\cdot \partial_x \psi(t,x)\; dx \; dt\\
%    &\qquad +\int_{\mathbb{R}^{d}}|u_0(x)-k|\psi(0,x)\; dx\ge 0.
%\end{split}
%    \end{equation}
 {\bf(N2)} For all $k\in \mathbb{R}$ and all nonnegative $\psi\in \mathcal{D}(\R\times\mathbb{R}^{N})$ there holds
\begin{equation}
\begin{split}\label{Panovw}
&\int_{\R_+\times\R^N}|\eta(x,v(t,x))-\eta(x,k)|\psi_{t}(t,x)\; dx \; dt\\
&\qquad+\int_{\R_+\times\R^N}\sgn(v(t,x)-k)(\bG(v(t,x))-\bG(k))\cdot \nabla \psi(t,x) \; dx \; dt\\
&\qquad +\int_{\R_+\times\R^N}\sgn (v(t,x)-k) f(t,x,\eta(x,v))\psi \;dx\;dt+\int_{\R^N}|u_0(x)-\eta(x,k)|\psi(0,x) dx\ge 0.
\end{split}
\end{equation}
%{\bf(N3)} For all $k\in \mathbb{R}$ and all nonnegative $\psi\in \mathcal{D}(\mathbb{R}^{d+1})$ there holds
%\begin{equation}
%\begin{split}\label{BGMSw}
%&\int_{\mathbb{R}^{d+1}_{+}}\sgn(U(g(t,x))-U(k))(\bG(U(g(t,x)))-\bG(U(k)))\cdot \nabla \psi(t,x) \; dx\; dt\\
%&\qquad +\int_{\mathbb{R}^{d+1}_{+}}|\eta(x,U(g(t,x)))-\eta(x,U(k))|\psi_{,t}\; dx \; dt\\
%&\qquad +\int_{\mathbb{R}^{d}}|u_0(x)-\eta(x,U(k))|\psi(0,x) \;dx\ge 0.
%\end{split}
%\end{equation}
\end{lemma}
%\begin{Rem}
%For the proof of the main result we only need the direction
%{\bf (N1)}$\Rightarrow$ {\bf (N3)}. For this direction in the proof of Lemma~\ref{equiv} we only need that  $U$ is Lipschitz.
%\end{Rem}
\begin{proof} %[Proof of Lemma~\ref{equiv}]
%\subsection{Kato inequality}
%The equivalence between {\bf (N1)} and  {\bf (N3)} can be proved in a different way. Namely, if we
%
%Showing the equivalence of {\bf (N2)} and {\bf (N3)} is obvious. Indeed, if for any $k$ we define $\tilde{k}:=U(k)$,  we can use $\tilde{k}$ in \eqref{Panovw} and by using the definition of $g$ we get exactly \eqref{BGMSw}. The opposite implication is proved in the same way since $U$ is one-to-one mapping.
To show {\bf (N1)} $\Rightarrow$ {\bf (N2)} consider
the equation
$$(u_i)_{t}+\div \Phi(x,u_i)=f_i(t,x,u_i),  \quad i=1,2.$$
For any two entropy weak  solutions $u_1,u_2$ the so-called Kato inequality holds
\begin{equation}\label{Kato}
\begin{split}
|u_1-u_2|_{t}+\div \left(\sgn(u_1-u_2)(\Phi(x,u_1)-\Phi(x,u_2))\right)\\
\le \sgn(u_1-u_2)(f_1-f_2)+|f_1-f_2|\chi_{\{u_1=u_2\}}
\end{split}
\end{equation}
in ${\mathcal D}'(\Rdp)$, cf.~Ref.~\cite{GaMa1996}.
Choosing in \eqref{Kato} $u_1=\eta(x,v_1)$ and $u_2=\eta(x,k)$ with $f_1= f,\, f_2\equiv 0$.
Note that the set of $k\in\R$ such that $|\{(t,x) : u_1(t,x)=\eta(x,k)\}|>0$ is at most countable and hence
it allows to pass from
\eqref{Kato} to {\bf (N2)}.

For showing the opposite direction let us consider  the problem with $f_i:\Rdp\times\R\to\R$ satisfying
Lipschitz condition with respect to the last variable.
\begin{equation}\label{eta-f}
\eta(x,v_i)_t+\div \bG(v_i)=f_i(t,x,\eta(x,v_i)), \quad i=1,2.
\end{equation}
For any entropy weak solutions $v_1, v_2$ in the sense of {\bf (N2)} it holds, cf.~\eqref{begin2}
\begin{equation}\label{Kato-eta}
\begin{split}
|\eta(x,v_1)-\eta(x,v_2)|_{t}+\div \left(\sgn(v_1-v_2)(\bG(v_1)-\bG(v_2))\right)\\
\le \sgn(v_1-v_2)(f_1-f_2)+|f_1-f_2|\chi_{\{v_1=v_2\}}
\end{split}
\end{equation}
in ${\mathcal D}'(\Rdp)$.
For passing from \eqref{Kato-eta} to {\bf (N1)} we choose again $u_1=\eta(x,v_1)$
and now  $v_2=\theta(x,k)$ with
$f_1= f$ and $f_2=\div \bG(v_2)=\div \Phi(x,k)$. Note again that the set of $k\in\R$ such that 
$|\{(t,x) : v_2(t,x)=\theta(x,k)\}|>0$ is at most countable and we pass to {\bf (N1)} what completes the proof.

%Although the inequalities {\bf (N1)} and {\bf (N2)} are not satisfied
%for all $k$, but for the set of values that is dense in~$\R$.
%The main effort in this method lies in proving the inequality \eqref{Kato-eta}, which may be done in a similar manner as the proof of averaged contraction principle.
%\end{Rem}
\end{proof}
%\begin{lemma}\label{Approx} {\tt Formulation and the proof will be changed}
%Let $U$ be a function prescribed by~\eqref{U}, hence in particular $U(0)=0$ and the Lipschitz constant 
%of $U$ is equal to one. Then $U^{-1}$ is a maximal (strictly) monotone operator, everywhere defined and multivalued.  We define $(U^{-1})_n$ as a Lipschitz function with a Lipschitz constant $n$ such that
% \begin{equation}
% |(U^{-1})_n(u)- \min((U^{-1})(u))|\le | L_n(u)-\min((U^{-1})(u))|
% \end{equation}
%  for any $u\in\R$ and every Lipschitz function $L_n(u)$ with a Lipschitz constant $n$.
%For   $x>0$
% \begin{equation}
% (U^{-1})_n(u)\le \min((U^{-1})(u)).
% \end{equation}
% Analogously, for $x<0$ 
% \begin{equation}
% (U^{-1})_n(u)\ge \max((U^{-1})(u)).
% \end{equation}
% Then
%% \begin{enumerate}
%% \item[(i)]
%% $(U^{-1})_n\to m((U^{-1})(u))$
%% \item[(ii)]
%$$ \|[(U^{-1})_n]^{-1}(g)-U(g)\|_{\infty,\R}\le\frac{1}{n}\sum\limits_{n=1}^\infty\frac{1}{n^2}.$$
%% \end{enumerate}
%\end{lemma}
%\begin{proof}
%
%\end{proof}
\begin{proposition}\label{monotone-conv}
Let $[a,b]\subset\R$ and let $f$ be continuous,  $f,f_n$ be monotone functions such that
$f_n\to f$ pointwisely. Then $f_n\to f$ uniformly on $[a,b]$. 
\end{proposition}
%\end{Appendix}
The above fact in an elementary exercise. For the proof see e.g.~\cite{Amann}.

\begin{proposition}\label{inverse-conv}
%Let $[a,b]\subset\R$ and 
Let $f_n:\R\to\R$,   ${\rm Im}\, (f_n)=\R,$  $f_n$ be strictly monotone functions. 
Let $f$ be a maximal monotone mapping with ${\rm Im}\, (f)=\R$ and let the inverse mapping $f^{-1}$ be continuous and
$f_n\to f$ a.e.. Then the inverse functions  converge locally uniformly to the inverse of the limit, namely
$(f_n)^{-1}\to f^{-1}$ uniformly on every compact subset of $\R$. 
\end{proposition}
\begin{proof}
We provide the proof by contradiction. Assume that $f_n$ converges a.e. to $f$ and that
$(f_n)^{-1}$ does not converge pointwisely to $f^{-1}$. Hence there exist $y, \varepsilon>0$ and
a~subsequence $(f_{n_k})^{-1}$ such that 
\begin{equation}\label{notin}
(f_{n_k})^{-1}(y)\notin[f^{-1}(y)-\bar\varepsilon,f^{-1}(y)+\bar\varepsilon].
\end{equation} 
for every $0<\bar\varepsilon<\varepsilon$.

\noindent
 {\it Case 1}:  $(f_{n_k})^{-1}(y)> f^{-1}(y)+\bar\varepsilon$.  Let $z=f^{-1}(y)$, hence $f(z)\ni y$ and there exists a selection 
$f^*$ such that $y=f^*(z)$. Define 
 $\bar y_{n_k}:=f_{n_k}^{-1}(f^*(z))$. By~\eqref{notin} we have the estimate 
$$\bar y_{n_k}>z+\bar\varepsilon.$$
Using the strict monotonicity of $f$ (hence obviously also $f^*$), monotonicity of $f_{n_k}$ and the definition of $\bar y_{n_k}$ we conclude an existence of 
$\delta$  such that for every $\bar\varepsilon\in(\frac{\varepsilon}{2},\varepsilon)$
\begin{equation}
0<\delta\le f^*(z+\bar\varepsilon)-f^*(z)=
f^*(z+\bar\varepsilon)-f_{n_k}(\bar y_{n_k})\le
f(z+\bar\varepsilon)-f_{n_k}(z+\bar\varepsilon).
\label{contra}\end{equation}
Since the number of points where $f$ is not continuous is  countable, it is always possible to choose such $\bar\varepsilon$ that $z+\bar\varepsilon$ is the point where $f$ is continuous (single valued). Hence for such $\bar\varepsilon$ \eqref{contra} contradicts the  convergence of $f_n$. 

\noindent
 {\it Case 2}:  $(f_{n_k})^{-1}(y)< f^{-1}(y)-\bar\varepsilon$.  Let again $z=f^{-1}(y)$, and  $y=f^*(z)$. Define 
 $\bar y_{n_k}:=f_{n_k}^{-1}(f^*(z))$ and observe that  
$$\bar y_{n_k}<z-\bar\varepsilon.$$
Again we conclude an existence of 
$\delta$  such that for every $\bar\varepsilon\in(\frac{\varepsilon}{2},\varepsilon)$
\begin{equation}
0<\delta\le f^*(z)-f^*(z-\bar\varepsilon)=
f_{n_k}(\bar y_{n_k})-f^*(z-\bar\varepsilon)\le
f_{n_k}(z-\bar\varepsilon)-f(z-\bar\varepsilon).
\label{contra2}\end{equation}
and we conclude in the same way as in the previous case. Hence $(f_n)^{-1}$  converges pointwisely to $f^{-1}$. 
The uniform convergence of $(f_n)^{-1}$ can be concluded by Proposition~\ref{monotone-conv}. 

\end{proof}

\bigskip

\noindent
{\bf Acknowledgement} \\
P. G. thanks to the National Science Centre, 
project no. 6085/B/H03/2011/40. The project of  
A. \'S.-G.  was financed by the  National Science Centre, DEC-2012/05/E/ST1/02218.
All the authors participate in the project Research Group Linkage of Alexander von Humboldt Foundation.

\bibliographystyle{abbrv}
\bibliography{hyper}

\begin{thebibliography}{10}

\bibitem{Amann}
H.~Amann and J.~Escher.
\newblock {\em {Analysis I.}}
\newblock {Basel: Birkh{\"a}user.}, 2005.

\bibitem{AnKaRi2010}
B.~Andreianov, K.~H. Karlsen, and N.~H. Risebro.
\newblock On vanishing viscosity approximation of conservation laws with
  discontinuous flux.
\newblock {\em Netw. Heterog. Media}, 5(3):617--633, 2010.

\bibitem{AudussePerthame}
E.~Audusse and B.~Perthame.
\newblock Uniqueness for scalar conservation laws with discontinuous flux via
  adapted entropies.
\newblock {\em Proc. Roy. Soc. Edinburgh Sect. A}, 135(2):253--265, 2005.

\bibitem{BaJe1997}
P.~Baiti and H.~K. Jenssen.
\newblock Well-posedness for a class of {$2\times2$} conservation laws with
  {$L^\infty$} data.
\newblock {\em J. Differential Equations}, 140(1):161--185, 1997.

\bibitem{BuGwSw2013}
M.~Bul{\'{\i}}{\v{c}}ek, P.~Gwiazda, and A.~{{\'S}}wierczewska Gwiazda.
\newblock Multi-dimensional scalar conservation laws with fluxes discontinuous
  in the unknown and the spatial variable.
\newblock {\em Math. Models Methods Appl. Sci.}, 23(3):407--439, 2013.

\bibitem{BuGwMaSw2011}
M.~Bul\'{\i}\v{c}ek, P.~Gwiazda, J.~M\'{a}lek, and A.~\'{S}wierczewska Gwiazda.
\newblock On scalar hyperbolic conservation laws with a discontinuous flux.
\newblock {\em Math. Models Methods Appl. Sci.}, 21(1):89--113, 2011.

\bibitem{Carrillo2003}
J.~Carrillo.
\newblock Conservation laws with discontinuous flux functions and boundary
  condition.
\newblock {\em J. Evol. Equ.}, 3(2):283--301, 2003.

\bibitem{Chen}
G.-Q. Chen, N.~Even, and C.~Klingenberg.
\newblock Hyperbolic conservation laws with discontinuous fluxes and
  hydrodynamic limit for particle systems.
\newblock {\em J. Differential Equations}, 245(11):3095--3126, 2008.

\bibitem{DiPerna}
R.~J. DiPerna.
\newblock Measure-valued solutions to conservation laws.
\newblock {\em Arch. Rational Mech. Anal.}, 88(3):223--270, 1985.

\bibitem{GaMa1996}
G.~Gagneux and M.~Madaune-Tort.
\newblock {\em Analyse math{\'e}matique de mod{\`e}les non lin{\'e}aires de
  l'ing{\'e}nierie p{\'e}troli{\`e}re}, volume~22 of {\em Math{\'e}matiques \&
  Applications (Berlin)}.
\newblock Springer-Verlag, Berlin, 1996.

\bibitem{Gi1993}
T.~Gimse.
\newblock {Conservation laws with discontinuous flux functions.}
\newblock {\em SIAM J. Math. Anal.}, 24(2):279--289, 1993.

\bibitem{GwSw2005}
P.~Gwiazda and A.~\'{S}wierczewska Gwiazda.
\newblock Multivalued equations for granular avalanches.
\newblock {\em Nonlinear Anal.}, 62(5):895--912, 2005.

\bibitem{Ji2011}
J.~Jimenez.
\newblock Mathematical analysis of a scalar multidimensional conservation law
  with discontinuous flux.
\newblock {\em J. Evol. Equ.}, 11(3):553--576, 2011.

\bibitem{KaRiTo2003}
K.~H. Karlsen, N.~H. Risebro, and J.~D. Towers.
\newblock {$L^1$ stability for entropy solutions of nonlinear degenerate
  parabolic convection-diffusion equations with discontinuous coefficients.}
\newblock {\em Skr., K. Nor. Vidensk. Selsk.}, 2003(3):49 p., 2003.

\bibitem{Kr70}
S.~N. Kru{\v{z}}kov.
\newblock First order quasilinear equations with several independent variables.
\newblock {\em Mat. Sb. (N.S.)}, 81 (123):228--255, 1970.

\bibitem{Mi2010}
D.~Mitrovic.
\newblock Existence and stability of a multidimensional scalar conservation law
  with discontinuous flux.
\newblock {\em Netw. Heterog. Media}, 5(1):163--188, 2010.

\bibitem{Panov}
E.~Y. Panov.
\newblock On existence and uniqueness of entropy solutions to the {C}auchy
  problem for a conservation law with discontinuous flux.
\newblock {\em J. Hyperbolic Differ. Equ.}, 6(3):525--548, 2009.

\bibitem{Ra03}
K.~R. Rajagopal.
\newblock On implicit constitutive theories.
\newblock {\em Appl. Math.}, 48(4):279--319, 2003.

\bibitem{Risebro}
N.~H. Risebro.
\newblock {An introduction to the theory of scalar conservation laws with
  spatially discontinuous flux functions.}
\newblock {Quak, Ewald (ed.) et al., Applied wave mathematics. Selected topics
  in solids, fluids, and mathematical methods. Berlin: Springer. 395-464.},
  2009.

\bibitem{Sz89a}
A.~Szepessy.
\newblock An existence result for scalar conservation laws using measure valued
  solutions.
\newblock {\em Comm. Partial Differential Equations}, 14(10):1329--1350, 1989.

\end{thebibliography}

\noindent
{\sc Piotr Gwiazda}\\
Institute of Applied Mathematics and Mechanics,\\
University of Warsaw, Banacha 2, 02-097 Warsaw, Poland\\
{\tt pgwiazda}@{\tt mimuw.edu.pl}\\[2ex]
{\sc Agnieszka \'Swierczewska-Gwiazda}\\
Institute of Applied Mathematics and Mechanics,\\
University of Warsaw,
Banacha 2, 02-097 Warsaw, Poland\\
{\tt aswiercz}@{\tt mimuw.edu.pl}\\[2ex]
 {
{\sc Petra Wittbold}\\
Faculty of Mathematics\\
University of Duisburg-Essen\\
45117 Essen\\
{\tt petra.wittbold}@{\tt uni-due.de}\\[2ex]
{\sc Aleksandra Zimmermann}\\
Faculty of Mathematics\\
University of Duisburg-Essen\\
45117 Essen\\
{\tt aleksandra.zimmermann}@{\tt uni-due.de}}

\end{document}